\documentclass[12pt,a4paper]{amsart}
\usepackage{amsmath,amssymb,latexsym}
\usepackage{amsthm}
\usepackage{bm}
\usepackage{overpic}
\usepackage{epstopdf}
\usepackage{color}
\usepackage{graphicx}     
\usepackage{caption}      
\usepackage{subcaption}
\usepackage{placeins}     
\theoremstyle{plain}
\newtheorem{theorem}{Theorem}[section]

\newtheorem{lemma}{Lemma}[section]
\theoremstyle{definition}
\newtheorem{definition}{Definition}[section]
\newtheorem{example}{Example}[section]
\theoremstyle{remark}
\newtheorem{remark}{Remark}[section]

\newcommand{\supp}{\mathop{\rm supp}}

\newcommand{\field}[1]{\mathbb{#1}}

\newcommand{\R}{\field{R}}

\newcommand{\N}{\field{N}}
\newcommand{\C}{\field{C}}

\renewcommand{\S}{\field{S}}

\newcommand{\B}{\field{B}}

\def\XXint#1#2#3{{\setbox0=\hbox{$#1{#2#3}{\int}$}
\vcenter{\hbox{$#2#3$}}\kern-.5\wd0}}

\author[Coroian]{Dan Coroian}
\address{Purdue University Fort Wayne \\ 2101 E. Coliseum Blvd. \\
	Fort Wayne, IN 46805 \\ USA}
\email{coroiand@purdue.edu}

\author[Dragnev]{Peter Dragnev}
\address{Purdue University Fort Wayne}
\email{pdragnev@purdue.edu}

\author[Legg]{Alan Legg}
\address{Purdue University Fort Wayne}
\email{leggar01@purdue.edu}

\author[Orive]{Ram\'{o}n Orive}
\address{Departmento de An\'{a}lisis Matem\'{a}tico \\ Universidad de La Laguna  38200 \\ The Canary Islands \\ Spain }
\email{rorive@ull.es}

\thanks{The research of Peter Dragnev was supported in part by the Lilly Endowment. \\ This work began in Spring 2025 while Ram\'{o}n Orive was a visiting Scholar in Residence at Purdue University Fort Wayne.}

\subjclass{31B05, 31C15, 41A10, 65D05, 78A30}

\keywords{Constrained energy problem, Balayage, Riesz $s$-Leja points, Saturation principle}

\title{On constrained Riesz minimum energy problems}

\date{\today}

\begin{document}

\maketitle



\vskip 1 cm

\vspace{1cm} 

\begin{abstract}

The constrained equilibrium problem for the logarithmic potential was introduced by Rakhmanov (1996) as he realized that the asymptotic distribution of the zeros of discrete orthogonal polynomials in a compact interval could be described in terms of the equilibrium measure of this interval in a class of measures subject to a certain constraint. Other authors such as Dragnev and Saff (1997), and Kuijlaars and Van Assche (1999), extended this approach to more general settings. These problems have proven to be useful for describing asymptotic distributions in different settings.

In the current paper, we consider constrained equilibrium problems in the Riesz setting, that is, for $s$-Riesz potentials in the hyperplane $\R^d\,,$ with $d\geq 1$ and $\max (0,d-2) < s < d$. Along with some general results, an illustrative example consisting of the solution of a constrained equilibrium problem in the unit ball is studied in detail. This is the main part of the paper. Finally, a number of numerical experiments show how the constrained equilibrium measure, the solution of the problem, may be discretized using the so-called \textit{constrained Leja points} introduced by Coroian and Dragnev (2001).


\end{abstract}

\section{Introduction}
\label{Intro}

The study of constrained equilibrium problems for logarithmic potentials started with E.A. Rakhmanov \cite{Rakh96}, who realized that the asymptotic distribution of zeros of discrete orthogonal polynomials was governed by a new type of equilibrium problem where the equilibrium measure must satisfy a certain constraint. Namely, given a compact set $K\subset \C$ and a positive measure $\sigma$ whose support contains $K$ and such that $\sigma(K) >1$, we consider the class of unit positive measures
\begin{equation}\label{constclass}
M^{\sigma} = \{\mu : \mu(K) =1 \;\text{and}\; 0<\mu \leq \sigma\}\,,
\end{equation}
where the inequality $\mu \leq \sigma$ means that $\sigma - \mu$ is a nonnegative measure. The upper bound $\sigma$ is called a {\em constraint}. 
As is customary, we will denote the logarithmic potential and logarithmic energy of a measure $\mu$ supported on $K$ with
\begin{equation}\label{logpot}
U^{\mu}(x) = -\,\int\,\log |x-t|\,d\mu(t), \quad I(\mu)=\int U^\mu (x)\, d\mu(x)\, . 
\end{equation}

Rakhmanov's idea was based on the fact that the zeros of discrete orthogonal polynomials are separated by the points where the discrete measure of orthogonality is supported. He applied this method to obtain the asymptotic distribution of the discrete Chebyshev polynomials (where the points are uniformly distributed on a compact interval, see \cite{Rakh96}). More precisely, Rakhmanov proved that for an admissible constraint $\sigma$ on the compact interval $E$, i.e., $\supp \sigma = E$, $\sigma (E) >1$, and such that $U^{\sigma}$, the logarithmic potential of $\sigma$, is continuous, the unique solution of the minimum energy problem 
$$\min\{I(\mu),\,\mu(E)=1, \mu \in M^{\sigma}\}$$ 
is a measure $\lambda^{\sigma}$ characterized by the following (Frostman's type) inequalities
\begin{equation}\label{Frost1}
U^{\lambda^{\sigma}}(x) \begin{cases}  \geq F^{\sigma},\, & x\in \supp(\sigma - \lambda^\sigma ), \\  \leq F^{\sigma},\, & x\in \supp (\lambda^{\sigma}), \end{cases}   
\end{equation}
for a real constant $F^{\sigma}$.

Observe that the continuity of $U^{\sigma}$ is assumed, which in turn implies the continuity of $U^{\lambda^{\sigma}}$ and, thus, the Frostman's type inequalities \eqref{Frost1} hold everywhere on their respective sets.

Rakhmanov \cite{Rakh96} applied this constrained equilibrium setting to a case where the support of the discrete measure of orthogonality consists of equally spaced points in the real interval $E=[-1,1]$; thus, the corresponding constraint is a multiple of the Lebesgue measure $\displaystyle d\sigma = C\,dx/2$, with $C>1$ (taking into account that $\|\sigma\|>1$ in order to be an admissible constraint). In this case, and others studied later, it is shown that on one part of the support of $\lambda^{\sigma}$ this measure satisfies $\lambda^{\sigma} < \sigma$ and, thus, the constraint has no effect, while on other parts (in particular, near the endpoints of the interval), $\lambda^{\sigma} = \sigma$; a comprehensive description of the different regions of the support can be found, for example, in \cite{BKMM07}.

Note that Rakhmanov \cite{Rakh96} proved his result, not just for orthogonal polynomials, but also for any sequence of monic extremal polynomials $\{p_n\}$ with respect to the discrete norm
\begin{equation}\label{norm1}
\|f\| = \left(\sum_{k=1}^N\,\eta_{k,n}\,|f(\zeta_{k,n})|^p\right)^{1/p},\,p>0\,,
\end{equation}
where $N = N(n)$, $\{\zeta_{k,n},\,k=1,\ldots,N\}$ are points in an interval $F$ chosen according to a certain asymptotic distribution $\sigma$ and satisfying a certain separation condition (see below), and the weights $\eta_{k,n},\,k=1\ldots,N,\,n=0,1,2,\ldots$ are asymptotically irrelevant, in the sense that
\begin{equation}\label{weight1}
\eta_{k,n}^{1/n} \longrightarrow 1 \, \mbox{as} \,\, n\rightarrow \infty\,,    
\end{equation}
uniformly in $k$. 

Dragnev and Saff \cite{DS97} extended Rakhmanov's setting to the weighted case, where for each $n$, we have  
$$\eta_{k,n} = w_n (\zeta_{k,n})\,,$$
for a certain sequence of continuous weights $\{w_n\}$ converging uniformly to a certain continuous weight $w$ on $F$. This modification of condition \eqref{weight1} transforms Rakhmanov's constrained energy problem from an unweighted setting to a weighted one. The weighted constrained energy problem allowed Dragnev and Saff \cite{DS97} to describe the asymptotic behavior of other classes of well-known discrete orthogonal polynomials, such as the Krawtchouk polynomials (see also \cite{DS00}). For discrete orthogonal polynomials on unbounded intervals, see \cite{KV99a,KV99b}.


Unlike the classical unconstrained setting, explicit solutions to these constrained energy problems are known only in a few cases. As a result, numerical methods have been proposed in the literature for approximate their solutions; see, for example, \cite{CD01,HV06}.

The purpose of this paper is to extend the study of these constrained minimum energy problems to the $s$-Riesz setting. More precisely, if $\mu$ is a measure in $\R^d, d \geq 1$, and $s\in (d-2,d)$, the $s$-Riesz potential of $\mu$ is defined by
\begin{equation}\label{Rieszpot}
U_s^{\mu}(x) :=\,\int\,\frac{d\mu(t)}{|x-t|^s}\,,
\end{equation}
while its $s$-Riesz energy is given by
\begin{equation}\label{Rieszenergy}
I_s(\mu) := \int U_s^{\mu}(x) d\mu(x) =\,\int\,\frac{d\mu(t) d\mu(x)}{|x-t|^s}\,.
\end{equation}

We consider the range $(d-2,d)$ for admissible values of parameter $s$, since in this case the kernel $K(x,y) = |x-y|^{-s}$ is strictly subharmonic. This allows us to employ several important tools in our analysis, most notably the first maximum principle \cite[Theorem 1.10]{Landkof} and techniques related to the notion of \textit{balayage} of a measure. For the definition and properties of balayage in the $s$-Riesz setting, see, for example, \cite[Sect. 3]{DOSW23}; see also Section 2 below. This case is commonly known in the literature as the \textit{Robin} case, while the case corresponding to the left endpoint of this interval, $s=d-2$, is commonly called the \textit{Coulomb} case. The logarithmic setting (or simply, the \textit{log} setting) is the limit case as $s\rightarrow 0^+$ and, thus, it agrees with the Coulomb case for $d=2$.

While in the log setting, the most common conductor used to pose minimum energy problems has been the real axis or even the whole complex plane (see \cite{ST24} as a general reference), in the $s$-Riesz case this role has been mainly played by spheres $\S \subset \R^d$ (see \cite{BDS2009} and \cite{BDS2014} to only cite a few); however, in some recent papers, the conductor has been a ball $\B \in \R^d$ (see \cite{Bilog} and \cite{DOSW25}), or even the whole hyperplane $\R^d$ (\cite{BDO}, \cite{DOSW23} and \cite{OW}). In Section 3, we consider the constrained minimum energy problem over a ball in $\R^d$. The solution of this problem is the main result of the paper.

The remainder of the paper is organized as follows. Section 2 introduces the constrained equilibrium problem in the general $s$-Riesz setting and establishes the main tools used throughout the paper, namely the \textit{Saturation Principle} (Theorem 2.1) and a reformulation of the original problem as a weighted (unconstrained) equilibrium problem (Theorem 2.2), together with the corresponding extension of the \textit{Balayage Representation}. In Section 3, a constrained problem on the unit ball $\B$ is completely solved. Finally, Section 4 presents numerical examples based on constrained Leja points.

\section{Constrained $s$-Riesz equilibrium problem}\label{Prel}

We now introduce the potential-theoretical preliminaries needed to formulate the constrained energy problem for Riesz $s$-potentials, for $\max (0,d-2)<s<d$. We will closely follow the logarithmic case described in Section \ref{Intro}. 

Let $K\subset \mathbb{R}^d$ be a compact set. Denote by $\mathcal{M}_K$ the set of probability Borel measures supported on $K$. Utilizing the Riesz $s$-energy definition \eqref{Rieszenergy} we define the {\em Wiener constant} associated with $K$ as
\begin{equation}\label{Wiener}
    W_K:=\inf \{I_s(\mu)\, :\, {\rm supp}(\mu)\subset K\}
\end{equation} and ${\rm cap}(K):=1/W_K$ as the {\em Riesz $s$-capacity} of $K$. When ${\rm cap}(K)>0$, there exists a unique measure $\mu_{K,s}=\mu_K\in \mathcal{M}_K$, such that $I_s(\mu_K)=W_K$, which we refer to as the {\em $s$-equilibrium measure}\footnote{In \cite[Chapter II]{Landkof}, Landkof calls $\mu_K$ the {\em minimizing $s$-measure}, while he refers to $\gamma_K={\rm cap}(K)\mu_K$ as the {\em equilibrium measure}.} of $K$. We will generally omit the dependence on $s$, unless we need to emphasize the role of this parameter. As in the logarithmic setting, the measure $\mu_K$ is characterized by the Frostman type characterization \eqref{Frost1} 

Here, {\em approximately everywhere} means that the inequality holds with the exception of a set of zero inner capacity. When $d-2 \leq s<d$, which Landkof calls the {\em Robin case} (see \cite[p. 138]{Landkof}, the maximum principle holds (see \cite[Theorem I.1.10]{Landkof}, namely, for any measure $\mu$, if the inequality  $U_s^\mu(x)\leq M$ holds $\mu$-a.e. then it holds everywhere. Therefore, in this case $U^{\mu_K}(x)\leq W_K$ on $\mathbb{R}^d$, which implies that $U^{\mu_K}(x)=W_K$ approximately everywhere on $K$. The compact sets $K$ that we consider will be regular (see \cite[Chapter V]{Landkof}, so there will be no exceptional sets.


\begin{example}
    In \cite[p. 163]{Landkof} a formula was derived for the equilibrium measure $\mu_K=\mu_R$ of $B(0,R)$, the ball of radius $R$ centered at the origin, namely
    \begin{equation}
    \label{EquilMeasR}d\mu_{R}=\frac{K_{R,s}}{\left( R^2-|x|^2\right)^{(d-s)/2}}\,dx, \quad \quad K_{R,s}=\frac{\Gamma(1+s/2)}{\,\Gamma (1-\alpha/2)\pi^{d/2}\,R^s},
    \end{equation}
where $dx$ denotes the restriction to $B(0,R)$ of the Lebesgue measure on $\mathbb{R}^d$ and $\alpha = d-s$.    
    
\end{example}

In contrast to the logarithmic setting, where constrained minimum-energy problems and their applications have been extensively studied, relatively few works have addressed the corresponding problem in the $s$-Riesz case; perhaps, \cite{Zorii} is the most important, since a Frostman-type characterization of the constrained equilibrium measure is established there (see \cite[Theorem 2.3]{Zorii}). In the Riesz setting this equilibrium problem with constraint also has important applications, such as those related to the so-called \textit{obstacle problem} (see e.g. \cite{CDM}, \cite{Donatella} and \cite{Gustafsson}).

In this article, we shall consider the constrained minimum energy problem for measures $\sigma$ with compact support and continuous potential and no external field, leaving the broader generality for a forthcoming paper. Namely, let $K$ be a compact subset of $\R^d$, $\max (0,d-2)<s<d$ and $\sigma$ an admissible constraint for $K$, that is, such that $U^{\sigma}$ is continuous, $\supp \sigma = K$ and $\sigma (K) > 1$. The continuity assumption implies that the $s$-energy of $\sigma$ is finite. Then, we see that the (unique) probability measure $\lambda^{\sigma}$ minimizing the energy \eqref{Rieszenergy} in the class $M^{\sigma}$ is characterized by the Frostman type condition
\begin{equation}\label{constFrostman}
U^{\lambda^{\sigma}}(x) \begin{cases}  \geq F^{\sigma},\, & x\in \supp(\sigma - \lambda^\sigma), \\  \leq F^{\sigma},\, & x\in \supp (\lambda^{\sigma}), \end{cases}   
\end{equation}
where it must be noticed that the requirement that $U^{\sigma}$ is continuous, along with the constraint $\lambda^{\sigma} \leq \sigma$ guarantees that $U^{\lambda^\sigma}$ is also continuous and, thus, the inequalities above hold everywhere on their respective sets.

It will be beneficial for the statement of our results to use the notation $\mathcal{FR}:=\supp (\sigma-\lambda^\sigma)$ for the {\em free region} and $\mathcal{SR}$ for the {\em saturated region}, where $\lambda^\sigma =\sigma$.

Our first result is the \textit{Saturation Principle}, which plays an essential role in the solution of the constrained $s$-Riesz minimum energy problem, as in the previously studied log setting (see \cite[Theorem 2.6]{DS97}).
\begin{theorem}\label{thm:satpple}
Let $\sigma$ be a positive measure with compact support $K \subset \R^d$ with non-empty interior and with continuous potential $U^{\sigma}$, $\max(0,d-2)<s<d$, such that $\sigma (K) > 1$.
Then, $\supp \lambda^{\sigma} = K$ and
\begin{equation}\label{SatP}(\lambda^{\sigma}) |_\mathcal{FR} \geq (\mu_K) |_\mathcal{FR}\,.\end{equation}

That is, $\lambda^{\sigma} = \sigma$ (the constraint is saturated) on any subset where the (signed) measure $(\mu_K - \sigma)$ is positive, i.e. $\mu_K$ violates the constraint $\sigma$.
\end{theorem}

\begin{proof} 
Let $\lambda^\sigma$ and $\mu_K$ denote the constrained and unconstrained minimizing $s$-measures. Recall that from \eqref{constFrostman} we have 
\begin{equation*}
U^{\lambda^{\sigma}}(x) \begin{cases}  \geq F^{\sigma},\, & x\in \mathcal{FR}, \\  \leq F^{\sigma},\, & x\in \supp (\lambda^{\sigma}), \end{cases}   
\end{equation*}
and from the Frostman conditions for $\mu_K$ we have \begin{equation*}\label{FrostRiesz}
-U^{\mu_K}(x) \begin{cases}  \leq -W_K,\, & {\rm approx. \ everywhere \ in }\  K, \\  \geq -W_K,\, & x\in \supp (\mu_K). \end{cases}   
\end{equation*}
As we are in the Robin case, by the Maximum Principle we can extend the inequality $U^{\lambda^\sigma}\leq F^\sigma$ to all of $\mathbb{R}^d$ and in particular $K$. Therefore, 
\begin{equation}\label{ConsrMinIneq} U^{\lambda^\sigma} (x)\leq U^{\mu_K}(x)+F^\sigma-W_K, \quad {\rm approx. \ everywhere \ in }\  K.\end{equation}
On the other hand, as $U^\sigma (x)$ is continuous, $U^{\lambda^\sigma}$ is continuous as well (it is both upper and lower semi-continuous), and hence $I(\lambda^\sigma)$ is finite. Therefore, sets of Riesz $s$-capacity zero will be also $\lambda^\sigma$-negligible, which implies that \eqref{ConsrMinIneq} holds $\lambda^\sigma$-a.e. 

By integrating the inequality with respect to $\lambda^\sigma$ we obtain that
\[ I(\lambda^\sigma)\leq \int U^{\mu_K}(x)\, d\lambda^\sigma (x)+F^\sigma - W_K.\]
As $\mu_K$ has finite energy, the inequality \eqref{ConsrMinIneq} holds $\mu_K$-a.e., so integrating it with respect to $\mu_K$ will yield the
\[ \int U^{\lambda^\sigma}(x)\, d\mu_K(x)=\int U^{\mu_K}(x)\, d\lambda^\sigma (x)\leq F^\sigma ,\]
where the equality follows from the Fubini-Tonelli theorem.
As $W_K=I(\mu_K)\leq I(\lambda^\sigma)$, we conclude that $F^\sigma-W_K\geq 0$, which implies that $U^{\mu_K}(x)+F^\sigma-W_k$ is $\alpha$-superharmonic. We can now apply \cite[Theorem 1.29]{Landkof} to conclude that  \eqref{ConsrMinIneq} holds everywhere in $\mathbb{R}^d$.

Applying the Maximum Principle again we conclude that $U^{\mu_K}(x)\leq W_K$ everywhere on $\mathbb{R}^d$, which yields
\begin{equation}\label{ConsrMinIneq2} U^{\lambda^\sigma} (x)\geq U^{\mu_K}(x)+F^\sigma-W_k \quad {\rm on} \quad \mathcal{FR}.\end{equation}

Utilizing the de la Vall\'{e}e-Poussin-type result \cite[Theorem 3.2]{DOSW23}, we combine \eqref{ConsrMinIneq} and \eqref{ConsrMinIneq2} to conclude \eqref{SatP}, which completes the proof.
\end{proof}

\begin{remark}
The condition that $K$ has a non-empty interior guarantees that every interior point of $K$ is in the support of its equilibrium measure and, hence, that $\supp \mu_K = K$ (see e.g. \cite{Wallin}).
\end{remark}

For the computation of the constrained equilibrium measure, it is essential to consider the reformulation of our constrained unweighted minimum energy problem as an equivalent unconstrained one but with external field. This is done in the following result, which extends, at least in part, in a natural way \cite[Theorem 2.13]{DS97}.

\begin{theorem}\label{thm:unconstr}
Let $\sigma, K$ and $s$ be as in Theorem \ref{thm:satpple}. Then, the positive measure $\tau = \sigma - \lambda^{\sigma}$ is such that 
\begin{equation}\label{tau}
\tau = (\|\sigma\|-1)\,\mu_Q\,,  
\end{equation}
where $\mu_Q = \mu_{K,Q}$ is the equilibrium measure on $K$ in the external field
\begin{equation}\label{extfi}
Q(x) = -\frac{U^{\sigma}(x)}{\|\sigma\|-1}\,.    
\end{equation}
\end{theorem}

\begin{remark}
    The support of $\sigma-\lambda^\sigma$, also referred to as the {\em free region}, plays a significant role. Although it is not known apriori, equation \eqref{tau} allows us to apply methods for determination of the support of the weighted equilibrium $\mu_Q$, such as the Mhaskar-Saff $\mathcal{F}$-functional minimization (see \cite[Theorem IV.1.5]{ST24} for the \textit{log} setting). Once $\supp (\sigma-\lambda^\sigma) = \mathcal{FR}$ is identified, the balayage representation formula \eqref{balyrep} below yields the solution of the constrained energy problem, $\lambda^\sigma$. 
\end{remark}
\begin{proof}
From the Frostman conditions \eqref{constFrostman} we have that $U^{\lambda^{\sigma}}(x)\leq F^{\sigma}$ for all $x\in \supp (\lambda^{\sigma})$. From the maximum principle, which holds for $s\in [\max(0,d-2),d)$, we can extend the inequality to all of $\mathbb{R}^d$. Combining with $U^{\lambda^{\sigma}}(x)\geq F^{\sigma}$ on $\mathcal{FR}$, we can re-write \eqref{constFrostman} for $\tau= \sigma - \lambda^{\sigma}$ as follows:
\begin{equation}\label{Frost_tau}
U^\tau(x) -U^\sigma (x)\, \begin{cases}  \geq -F^\sigma,\, & x\in \mathbb{R}^d\,, \\   \leq -F^{\sigma},\,& x\in \supp (\tau)=\mathcal{FR}\, . \end{cases}
\end{equation}
Dividing by $\| \sigma \|-1$ we obtain that $\tau/(\| \sigma\|-1)$ is a probability measure satisfying the Frostman conditions for the external field problem on $\mathbb{R}^d$ (see e.g. \cite[Theorem 2.1]{DOSW23}) with an external field $Q(x)=-U^\sigma (x)/(\|\sigma\|-1)$. This implies the relation \eqref{tau} and completes the proof.
\end{proof}

Next, we will derive a \textit{balayage representation} of $\lambda^{\sigma}$, extending the result in \cite[Theorem 2.13 and Corollary 2.15]{DS97} (for the logarithmic potential setting) to the general $s$-Riesz context. 
To do it, let us briefly recall the notion of \textit{balayage} or \textit{sweeping out} of a measure onto a closed set for $\max(0,d-2)\leq s<d$ (see \cite[Chapter IV, p. 264]{Landkof}).

Let us now recall the notion of \textit{balayage} of a measure when $d-2 \leq s <d$, see \cite[Chapter IV, p.264]{Landkof}.
Namely, given a closed set $K\subset \R^d$ of positive capacity and a {positive} measure $\nu$ \textcolor{black}{of finite total mass}, 
there exists a unique {positive} measure
$$\widehat{\nu}:=
Bal(\nu,K)$$
called the Riesz $s$-balayage of $\nu$ onto $K$ satisfying the following: \\
1) $S_{\widehat{\nu}} \subseteq K$, \\
2) $\widehat{\nu}$ is zero on the set of irregular points of \textcolor{black}{the complement of $K$}, and
\begin{equation}\label{balayage}
U^{\widehat{\nu}}(x) = U^{\nu}(x)\text{ q.e.\ on }K,\qquad U^{\widehat{\nu}}(x) \leq U^{\sigma}(x)\text{ on }\R^d.
\end{equation}
The mass of $\hat\nu$ satisfies $\|\hat\nu\|\leq\|\nu\|$, i.e.\ the balayage may induce a \textit{mass loss}, in particular when $K$ is bounded. 

In \cite[Theorem 2.13]{DS97}, the previous relation \eqref{extfi} allows us to obtain a \textit{Balayage Representation} of the constrained equilibrium measure, which greatly helps in the computation of the density of $\lambda^{\sigma}$. In the $s$-Riesz setting, as said above, the balayage or sweeping out of a measure onto a compact set entails a mass loss. Hence, it is necessary to cope with this important difficulty. In this sense, if we denote by $\widehat\sigma$ the balayage of $\sigma$ on $\supp (\sigma-\lambda^\sigma)=\mathcal{FR}$ and by $m(\widehat \sigma)$ its mass, we have that $0<m(\widehat \sigma)<\|\sigma\|<1$. The following theorem is direct consequence of Theorem \ref{thm:unconstr}. 
\begin{theorem}
Let $\sigma, K$ and $s$ be as in Theorem \ref{thm:satpple}. Then the following balayage representation holds
    \begin{equation}\label{balyrep}
    \lambda^{\sigma} = \sigma - \widehat \sigma + 
    (m(\widehat \sigma) +1 - \|\sigma\|) \mu_{\mathcal{FR}}=\sigma - \widehat \sigma +\frac{F^\sigma}{W_{\mathcal{FR}}} \mu_{\mathcal{FR}} \,,
\end{equation}
where $\mu_{\mathcal{FR}}$ denotes the equilibrium measure of the free region $\mathcal{FR}$. 
\end{theorem}

Compare this identity with \cite[(2.17)]{DS97}.

\begin{proof}
    From \eqref{Frost_tau} we have that $U^\tau (x)+F^\sigma \leq U^\sigma (x)$ on $\mathbb{R}^d$ with equality holding on $\mathcal{FR}$. Since $s\in (d-2,d)$, we have that $U^{\mu_{\mathcal{FR}}}(x)=W_{\mathcal{FR}}$ for all $x\in \mathcal{FR}$, which implies that the measure $\tau+(F^\sigma/W_{\mathcal{FR}})\mu_{\mathcal{FR}}$ is the balayage ${\rm Bal}(\sigma,\mathcal{FR})$ of $\sigma$ onto $\mathcal{FR}$. Thus, representation \eqref{balyrep} follows.
\end{proof}

To end this section, we recall another notion that will play a crucial role in the next section: the \textit{signed equilibrium measure}. Namely, we have the following
\begin{definition}
Let $\Sigma$ be a closed subset of $\R^d$. A {\it  signed equilibrium measure} for $\Sigma$ in the external field $Q$ is a (finite) signed measure $\eta_{Q,\Sigma}$ with finite $s$-energy supported on $\Sigma$ such that $\eta_{Q,\Sigma} (\Sigma) = 1$ and there exists a finite constant $C_{Q,\Sigma}$ such that
\begin{equation}\label{defsigned}
U^{\eta_{Q,\Sigma}}(x) + Q(x) = C_{Q,\Sigma}\quad\text{q.e. on }\Sigma.
\end{equation}
\end{definition}

If this signed equilibrium measure exists, then it is unique, see \cite[Lemma 23]{BDS2009}.
\textcolor{black}{Also, recall that a signed measure $\nu$ has finite energy if and only if both $\nu^+$ and $\nu^-$ have finite energy
(see e.g. \cite[Definition 4.2.4, p.134]{BHS} for more general kernels). Thus, it follows from our definition that both $\eta_{Q,\Sigma}^+$ and $\eta_{Q,\Sigma}^-$ have finite energy.}

The main advantage of using the signed equilibrium measure is that it provides information about the support of the (positive) equilibrium measure; namely, the latter is included in the support of the positive part of the signed equilibrium measure. This fact will be used in the next section (see Lemma \ref{Lem:DOSW}).

\section{A constrained energy problem on the Unit Ball.}

This section is devoted to solving a constrained $s$-Riesz minimum energy problem on the unit ball $\B$ in $\R^d$, for $\max(0,d-2)<s<d$. The constraint will be a multiple of the Lebesgue measure in the ball $\B$, that is,
\begin{equation}\label{constraint}
d\sigma (x) = C dx,\,x\in \B\,,    
\end{equation}
where $\displaystyle C > \frac{\Gamma (1+d/2)}{\pi^{d/2}}$, to guarantee that $\|\sigma\| > 1$ and, then, the admissibility of the constraint.

This example is a natural extension of the example considered by Dragnev and Saff in \cite[Example 4.4]{DS97} for the log setting, where the constraint is a multiple of the Lebesgue measure on the Unit Disk in $\C$, extending in turn Rakhmanov's original setting in the real interval $[-1,1]$. Notice that our analysis is also valid for $d=1$ and $s\in (0,1)$ and, hence, also extends Rakhmanov's analysis in \cite{Rakh96}.

First of all, from \eqref{extfi} we have $\displaystyle \tau = (\|\sigma\|-1)\,\mu_Q\,,$ where $\mu_Q$ is the equilibrium measure of $\B$ in the external field 
\begin{equation}\label{extficonstr}
Q(x) = -\frac{U^{\sigma}(x)}{\|\sigma\|-1}\,.
\end{equation}
Here, 
\begin{equation}\label{C}
\displaystyle \|\sigma\| = \frac{C \pi^{d/2}}{\Gamma (1+d/2)} > 1.
\end{equation}
Thus, our first concern is to compute the equilibrium measure in the Unit Ball in the external field $Q$ given by \eqref{extficonstr}. Our method will consist in computing the corresponding signed equilibrium measure $\eta_{Q,B_r}$ for any radius $r\in (0,1)$ using the balayage measures.
That is, it is easy to check that
\begin{equation}\label{seqconstr}
\eta_{Q,B_r}(x) = \,\frac{1}{\|\sigma\|-1}\,Bal (\sigma, B_r)(x)\,+\,\left(1 - \,\frac{m_r}{\|\sigma\|-1}\right)\,\mu_r(x)\,,   
\end{equation}
where $Bal (\sigma, B_r)$ denotes the balayage of $\sigma$ onto the ball $B_r$, and $m_r \in (0,\|\sigma\|)$ stands for the mass of this balayage (recall that the balayage onto a set generally entails a mass loss, in particular when this is a bounded set); the density of $\mu_r$ is given in \eqref{EquilMeasR}.

Since we need to sweep out a measure supported in $\B$ onto a subset $B_r, r\in (0,1)$, we can apply the superposition principle \cite[Eq. (4.5.6)]{Landkof} and hence we have 
\begin{equation}\label{superposition}
Bal (\sigma, B_r) = \sigma |_{B_r}\,+\,Bal \left(\sigma |_{A_{r,1}}, B_r\right)\,,
\end{equation}
where $A_{r,1}$ is the multidimensional annulus (or shell) $\B \setminus B_r$. Thus, taking into account the expression of $\sigma$, \eqref{superposition} may be rewritten in the form
\begin{equation*}\label{balsigma0}
dBal (\sigma, B_r) = C\,\left(1 + \int_{|t|=r}^{|t|=1}\,\widehat{\delta}_t(x)\,dt \right)\,dx\,,
\end{equation*}
where $\widehat{\delta}_t$ stands for the balayage of the Dirac delta $\delta_t$, with $r<|t|\leq 1$, onto $B_r$. 
Now, using \cite[pp. 121-122]{Landkof} for the density of $\delta_t$ above, we obtain the following expression for the density of the balayage of $\sigma$ onto $B_r$.
\begin{equation}\label{balsigma}
\begin{split}
Bal' (\sigma, B_r) (x) & =  C\,\left(1 + \,\frac{\Gamma (d/2)\, \sin (\pi \alpha /2)}{\pi^{1+d/2}\,(r^2-|x|^2)^{\alpha/2}}\,\int_{|t|=r}^{|t|=1}\,\frac{(|t|^2-r^2)^{\alpha/2}}{|t-x|^d}\,dt \right) \\
& := C\,\left(1 + \,\frac{\Gamma (d/2)\, \sin (\pi \alpha /2)}{\pi^{1+d/2}\,(r^2-|x|^2)^{\alpha/2}}\,I(x,r)\right)\,.
\end{split}
 \end{equation}
Recall that $\alpha = d-s \in (0,2).$

Now, we turn to simplify the integral $I(x,r)$ in \eqref{balsigma}. Indeed, making the change to polar coordinates in $\R^d$ and proceeding analogously as in \cite[Appendix, p. 400]{Landkof}, we get the expression
\begin{equation}\label{I}
I(x,r) = \,\frac{2\pi^{d/2}}{\Gamma (d/2)}\,\int_r^1\,\frac{(\rho^2-r^2)^{\alpha/2} \rho d\rho}{\rho^2-|x|^2}\,,    
\end{equation}
which after making the change $\rho^2-r^2 = (1-r^2)v$ and using the identity \cite[Eq. (15.3.1)]{Abr}, allows us to rewrite \eqref{I} in the form
\begin{equation}\label{I(x,r)}
\begin{split} I(x,r) & = \frac{\pi^{d/2} (1-r^2)^{1+\alpha/2}}{\Gamma(d/2))}\;\int_0^{1}\,\frac{v^{\alpha/2}}{(1-r^2)v+r^2-|x|^2}\,,\\
& = \frac{2 \pi^{d/2} (1-r^2)^{1+\alpha/2}}{(\alpha+2) \Gamma (d/2) (r^2-|x|^2)}\;_2 F_1 \left(1, 1+\alpha/2; 2+\alpha/2; - \frac{1-r^2}{r^2-|x|^2}\right)\,,
\end{split}
\end{equation}
where, as usual, $_2 F_1 (a,b;c;d)$ denotes the hypergeometric function.

Using \eqref{balsigma} and \eqref{I(x,r)}, we can state the following
\begin{lemma}\label{lem:bal}
The balayage of the measure $\sigma$ in \eqref{constraint} onto a ball $B_r$, centered at the origin with radius $0<r<1$, has a density given by
\begin{equation}\label{balay}
\begin{split}
& Bal'(\sigma, B_r)(x) = C\,\left(1 + \,\frac{\sin (\pi \alpha /2) (1-r^2)^{1+\alpha/2}}{\pi (r^2-|x|^2)^{\alpha/2}}\;
\int_0^{1}\,\frac{v^{\alpha/2} dv}{(1-r^2)v+r^2-|x|^2}\,\right) = \\
& C\,\left(1 + \,\frac{2\sin (\pi \alpha /2) (1-r^2)^{\alpha/2}}{(\alpha +2) \pi (r^2-|x|^2)^{1+\alpha/2}}\;_2 F_1 \left(1, 1+\alpha/2; 2+\alpha/2; - \frac{1-r^2}{r^2-|x|^2}\right)\right)\,.   
\end{split}
\end{equation}

\end{lemma}

Next, we will apply a similar method to that followed in \cite{DOSW23}-\cite{DOSW25} to find the (positive) equilibrium measure, that is, to determine the critical radius $r_0$ for which the density of the signed equilibrium measure vanishes on the boundary (the sphere of radius $r_0$).
First, we will figure out the small critical radius for which
\begin{equation*}\label{condit}
\lim_{|x|\rightarrow r^-}\,(r^2-|x|^2)^{\alpha/2}\,\eta'_{Q,B_r}(x) = 0\,,  
\end{equation*}
which, taking into account \eqref{balsigma} and \eqref{balay}, leads to the identity
\begin{equation}\label{critradident}
\left(1-\,\frac{m_r}{\|\sigma\|-1}\right)\,K_{r,s} = -\frac{2C \sin(\alpha \pi/2)}{\alpha \pi (\|\sigma\|-1)}\,(1-r^2)^{\alpha/2}\,,
\end{equation}
where the constant $K_{r,s}$ is given in \eqref{EquilMeasR}. Then, defining
\begin{equation}\label{eqcritrad}
    G(r):= K_{r,s} \alpha \pi (\|\sigma\|-1 - m_r) + 2C \sin(\alpha \pi/2) (1-r^2)^{\alpha/2}\,,
\end{equation}
and recalling that $\alpha \in (0,2), m_r \in (0,\|\sigma\|)$ with $\lim_{r\rightarrow 0^+}\,m_r = 0$ and $\lim_{r\rightarrow 1^-}\,m_r = \|\sigma\|$, as well as the fact that the constants $K_{R,s}$ in Eq. \eqref{EquilMeasR} and $C$ in Eq. \eqref{C} are positive, we have $G(0)>0$ and $G(1)<0$. In addition, it is easy to see that $G$ is strictly decreasing on the interval $(0,1)$. Therefore, we conclude that there exists a unique $r_0 \in (0,1)$ such that $G(r_0) =0$.





Now, we are going to prove that $B_{r_0}$ is the support of the equilibrium measure $\mu_Q$. To do it, consider the signed equilibrium $\eta_{B_{r_0}}$. From \eqref{seqconstr}, and since $G(r_0)=0$, with $G$ defined in \eqref{eqcritrad}, we have the following expression of its density when $|x|<r_0$.
\begin{equation}\label{semr0}
\begin{split}
& \eta'_{Q,r_0}(x) = \,\frac{1}{\|\sigma\|-1}\,Bal' (\sigma, B_{r_0})\,+\,\left(1 - \,\frac{m_{r_0}}{\|\sigma\|-1}\right)\,\gamma_{r_0} = \\
& \frac{C}{\|\sigma\|-1}\,\left(1-\, \frac{\sin (\pi \alpha/2) (1-r_0^2)^{\alpha/2} (r_0^2-|x|^2)^{1-\alpha/2}}{\pi}\,\int_0^{1}\,\frac{v^{\alpha/2} dv}{(1-r_0^2)v+r_0^2-|x|^2}\right)\,.
\end{split}
\end{equation}

We must prove that the density \eqref{semr0} is positive for $|x|\leq r_0$. This will be shown by the following
\begin{lemma}\label{lem:positive}
The density of the signed equilibrium measure \eqref{semr0}, $f(r) = \eta'_{Q,r_0}(r)$ is nonnegative for $r=|x|\leq r_0$.
Moreover, $\lim_{r\rightarrow r_0^-}\,f(r) = 0.$
\end{lemma}
\begin{proof}
Making the change of variable $\displaystyle v = \frac{r_0^2-r^2}{(1-r_0^2) u}$ in the integral in \eqref{semr0}, we have
$$\int_0^{1}\,\frac{v^{\alpha/2} dv}{(1-r_0^2)v+r_0^2-|x|^2}\,= \frac{(r_0^2-|x|^2)^{\alpha/2-1}}{(1-r_0^2)^{\alpha/2}}\,\int_\frac{r_0^2-|x|^2}{1-r_0^2}^{\infty}\,\frac{t^{-\alpha/2}}{1+t}\,dt\,.$$
Thus, the well-known identity 
$$\int_0^{\infty}\,\frac{t^{-\alpha/2}}{1+t}\,dt = B(\alpha/2,1-\alpha/2) = \Gamma (\alpha/2) \Gamma (1-\alpha/2) = \,\frac{\pi}{\sin(\alpha \pi/2)}$$
shows that
$$\int_0^{1}\,\frac{v^{\alpha/2} dv}{(1-r_0^2)v+r_0^2-|x|^2}\, = \,\frac{(r_0^2-|x|^2)^{\alpha/2-1}}{(1-r_0^2)^{\alpha/2}}\,\int_\frac{r_0^2-|x|^2}{1-r_0^2}^{\infty}\,\frac{t^{-\alpha/2}}{1+t}\,dt\,\leq \, \frac{(r_0^2-|x|^2)^{\alpha/2-1}}{(1-r_0^2)^{\alpha/2}}\,\, B(\alpha/2,1-\alpha/2) \,,$$
and this proves that $f(r) \geq 0$, for $r\leq r_0$.

Now, using the same change of variable, we easily get
$$\lim_{r\rightarrow r_0^-}\,(r_0^2-|x|^2)^{1-\alpha/2}\,\int_0^{1}\,\frac{v^{\alpha/2} dv}{(1-r_0^2)v+r_0^2-|x|^2}\, = \,\frac{(1-r_0^2)^{\alpha/2})}{B(\alpha/2,1-\alpha/2)}\,,$$
which proves that $\lim_{r\rightarrow r_0^-}\,f(r) = 0\,.$

\end{proof}

The result in Lemma \ref{lem:positive} is illustrated in Fig. \ref{Fig:eta} which shows the graph of $f(r) = \eta'_{Q,r_0}(r),\,r\in[0,r_0]\,,$ in some cases. The density function is clearly positive and vanishes on the boundary.

\begin{figure}[!htbp]
    \centering
    \begin{subfigure}[b]{0.35\textwidth}
        \centering
        \includegraphics[width=\textwidth]{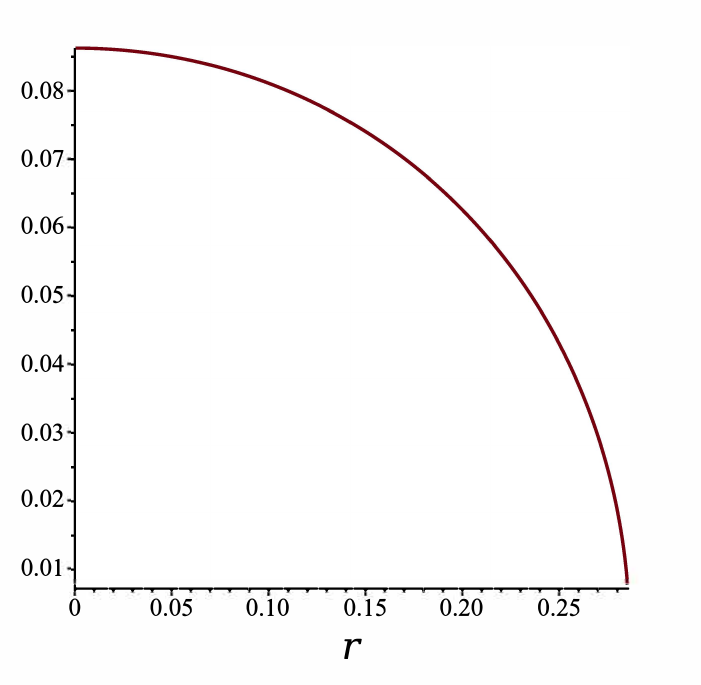}
        \caption{Case $d=2, s=1$ and $C=1$.}
        \label{Fig:eta1}
    \end{subfigure}
   \hspace{0.1\textwidth}
    \begin{subfigure}[b]{0.35\textwidth}
        \centering
        \includegraphics[width=\textwidth]{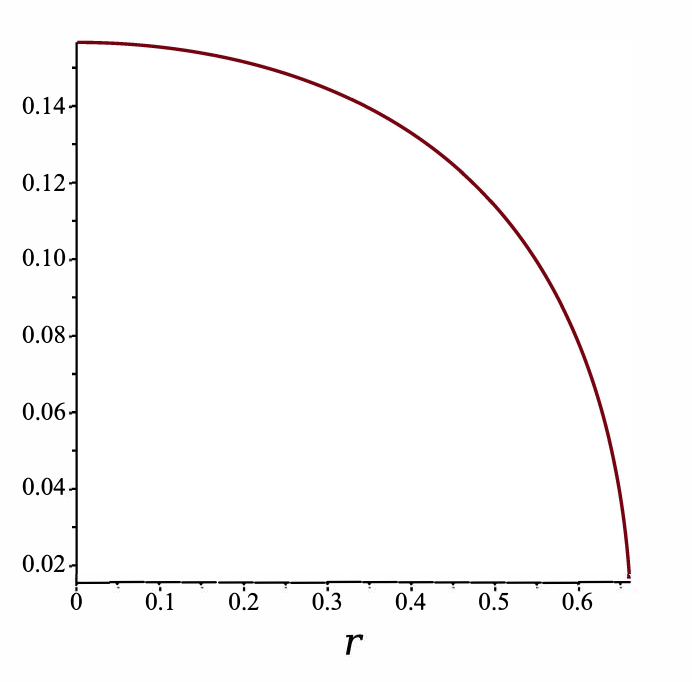}
        \caption{Case $d=2, s=1$ and $C=5$.}
        \label{Fig:eta5}
    \end{subfigure}
    \caption{The density of $f(r), \,r\in[0,r_0].$}
    \label{Fig:eta}
\end{figure}


Now, our concern is to prove that the signed equilibrium measure \eqref{semr0} is actually the equilibrium measure $\mu_Q$. 
In this sense, take into account that the procedure employed to figure out the critical radius $r_0$ implies that for any $r>r_0$, the signed equilibrium measure $\eta_{Q,r}$ has a negative component close to the boundary; then \cite[Lemma 3.15]{DOSW23} applies. For the sake of completeness, we repeat here the statement of this result.
\begin{lemma}\label{Lem:DOSW}
Let $d-2\leq s <d$ and let $\Sigma$ be a closed subset of $\R^d$ \textcolor{black}{of positive capacity}.
Assume an equilibrium measure $\mu_{Q}$ and a signed equilibrium measure $\eta_{Q,\Sigma}$ exist. Denote by $\eta_{Q,\Sigma}^+$ the positive part in the Jordan decomposition of $\eta_{Q,\Sigma}$.
 Then,
\\[10pt]
(i) one has
\begin{equation}\label{signeddomin}
\mu_Q \leq \eta_{Q,\Sigma}^+\, ;
\end{equation}
in particular,
\begin{equation*}
S_{\mu_Q} \subseteq S_{\eta_{Q,\Sigma}^+}.
\end{equation*}
(ii) \textcolor{black}{Let $\Sigma_{1}$ be a closed subset of $\Sigma$ that admits a signed equilibrium measure $\eta_{Q,\Sigma_1}$} for which $S_{\mu_Q}\subset\Sigma_{1}$. If $\eta_{Q,\Sigma_{1}}$ is a positive measure, then $\mu_{Q}=\eta_{Q,\Sigma_{1}}$.
\\[10pt]
(iii) Let $\Sigma=\R^{d}$ and let $Q$ be an external field such that $S_{\mu_Q}$ is compact. Assume that there exists an $R_{0}>0$ such that, for each $R$ larger than $R_{0}$, \textcolor{black}{there exists a neighborhood $V$ of the boundary of $B_{R}$, such that the restriction of the signed equilibrium measure $\eta_{Q,B_{R}}$ to $V$ is negative.} Then $S_{\mu_Q}\subset B_{R_{0}}$.
\end{lemma}

Observe that, indeed by Lemma \ref{Lem:DOSW}, $B_{r_0}$ turns out to be the support of the (positive) equilibrium measure $\mu_Q$ and
$$\mu_Q = \eta_{Q,r_0}\,.$$

Now, taking into account \eqref{extficonstr}, \eqref{semr0} and the previous discussion, we are in a position to state our main result in this section. 
\begin{theorem}\label{thm:sect3}
Let $d-2<s<d$. Then, the density of the constrained equilibrium measure $\lambda^{\sigma}$ for the unit ball $\B$ with the constraint $\sigma$ given by \eqref{constraint}, is given by
\begin{equation}\label{constreqmeas}
(\lambda^{\sigma})'(x) =  \begin{cases}
     C \,\frac{\sin (\pi \alpha/2) (1-r_0^2)^{\alpha/2} (r_0^2-|x|^2)^{1-\alpha/2}}{\pi}\,\int_0^{1}\,\frac{v^{\alpha/2} dv}{(1-r_0^2)v+r_0^2-|x|^2},\; & |x|\leq r_0,
     \\[.5cm]
    C,\;& r_0\leq |x|\leq 1,
 \end{cases}   
\end{equation}
where $r_0\in (0,1)$ is the unique root of the equation $G(r) = 0$, with $G$ defined in \eqref{eqcritrad}.

\end{theorem}
\begin{remark}\label{rem:vanishing}
    In Lemma \ref{lem:positive} it was shown that the density of the equilibrium measure $\mu_Q$ vanishes in the boundary sphere $|x|=r_0$. This behavior agrees with that found in \cite{BDO} and \cite{DOSW23}, and with what is known for the \textit{log} setting in the real axis \cite{KuML}. Therefore, it is reasonable to conjecture that this is a generic behavior for weighted equilibrium problems for Riesz Potentials in the Robin setting ($\max(0,d-2)<s<d$). 
    Otherwise, as for the Coulomb case ($s=d-2$), in \cite{OW} it is shown that the weighted equilibrium measure does not necessarily vanish on the boundary of its support.
\end{remark}

\begin{remark}\label{computr0}
Next, we will show how to compute the critical radius $r_0$ in an efficient way. To do it, we use the formula for the mass of the balayage of a pointwise charge placed outside the ball $B_{r}$ onto that ball (see \cite[Formula (4.5.6')]{Landkof},
$$Bal (\delta_y,B_{r})(\R^d) = \,\frac{U^{\gamma_{r}}(y)}{W(B_{r})}\,=\,\frac{2r^s\,_2 F_1 \left(s/2,\alpha/2;d/2;\frac{r^2}{|x|^2}\right)}{s B(s/2,\alpha/2) |y|^s},$$
where the value of the energy $W(B_{r_0})$ was taken from \cite[Section 4.6]{BHS}.

Then, by the superposition principle, we have that
\begin{equation}\label{mr}
m_r = Bal (\sigma,B_{r})(\R^d) = C \left(V(B_{r}) + \,\frac{2r^s\,\int_{r\leq |y|\leq 1}\,\frac{_2 F_1(s(2,\alpha/2;d/2;r^2/|y|^2)}{|y|^s}\,dy}{s B(s/2,\alpha/2)}\right)\,.
\end{equation}
Substituting \eqref{mr} into Eq. \eqref{eqcritrad}, we obtain an equation that involves the integral of a hypergeometric function; however, given $C, d$ and $s$ (or $\alpha$), the critical value of the radius $r_0$ can be calculated numerically using a mathematical software (Maple, for example). 

\end{remark}

\begin{remark}\label{Coulomb}
Observe that as $\alpha \rightarrow 2^{-}$, that is, as we approach the Coulomb setting, the value of the density of $\lambda^{\sigma}$ in \eqref{constreqmeas} approaches zero, for $|x|<r_0$. This means that the support of the constrained equilibrium measure converges to the annulus $r_0\leq |x|\leq 1$. This result agrees with what was observed in \cite[Example 4.4]{DS97}, where the logarithmic setting was considered in $\C$ (i.e., the Coulomb setting for $d=2$). 
\end{remark}

To end this section, we illustrate the results obtained above for the critical radius $r_0$ and the density of the constrained equilibrium measure $\lambda^{\sigma}$ for some values of $C,d,s$. Recall that we set $\alpha = d-s$.

First, we consider the case $d=2, s=1$ and take $C=1$. Recall that this value of $C$ is bigger than its lower bound $1/\pi$. In this case, the critical radius computed by \eqref{eqcritrad} is $r_0 = 0.2861649908$ and, thus, the saturated region is large.  Next, for the same values of $d$ and $s$, take $C=5$. The higher the value of the constant, the smaller the saturated region. Indeed, in this case $r_0 = 0.6622337642$. The two cases are shown in Figure \ref{Fig:radial}.

\begin{figure}[!htbp]
    \centering
    \begin{subfigure}[b]{0.35\textwidth}
        \centering
        \includegraphics[width=\textwidth]{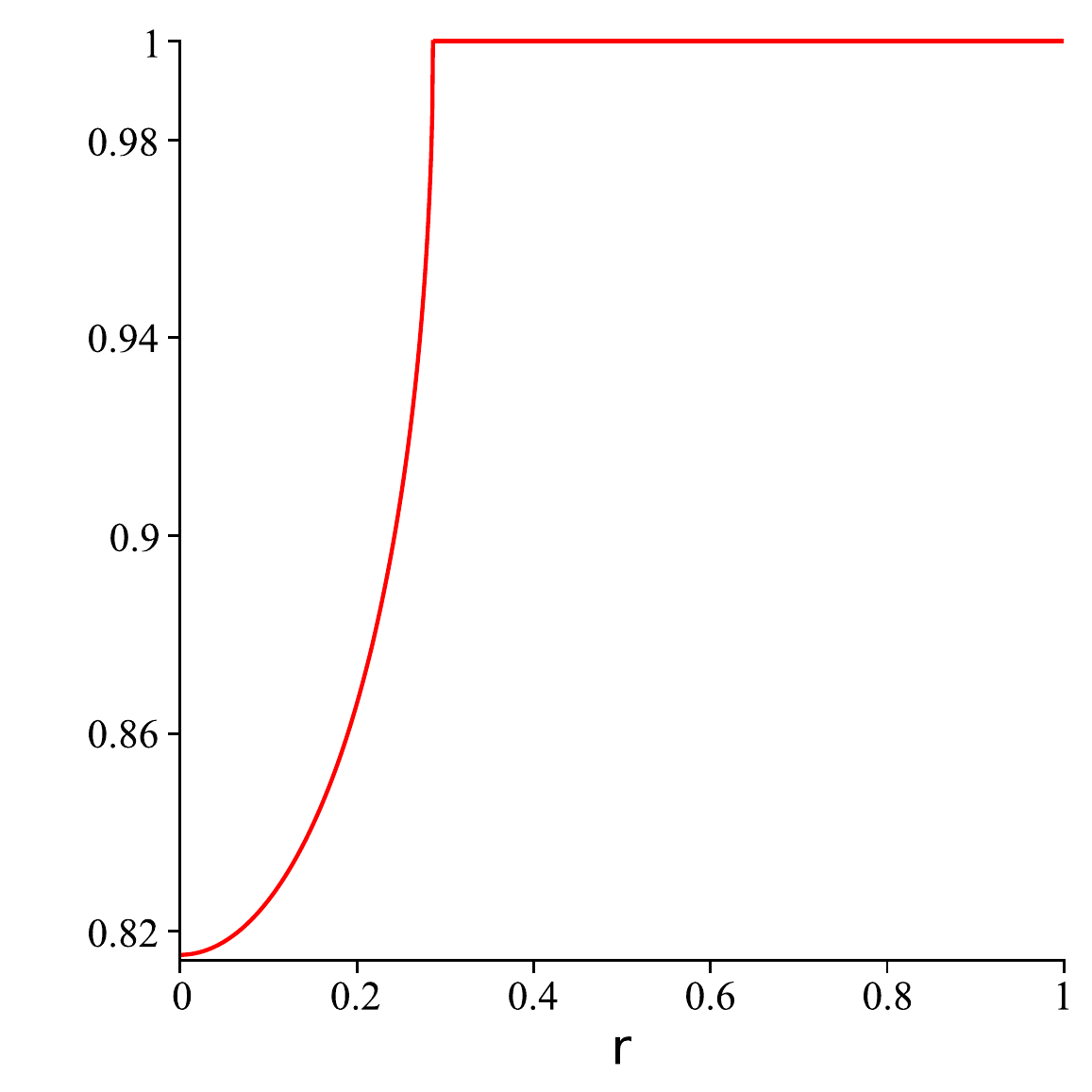}
        \caption{Case $d=2, s=1$ and $C=1$.}
        \label{Fig:radiala}
    \end{subfigure}
   \hspace{0.1\textwidth}
    \begin{subfigure}[b]{0.35\textwidth}
        \centering
        \includegraphics[width=\textwidth]{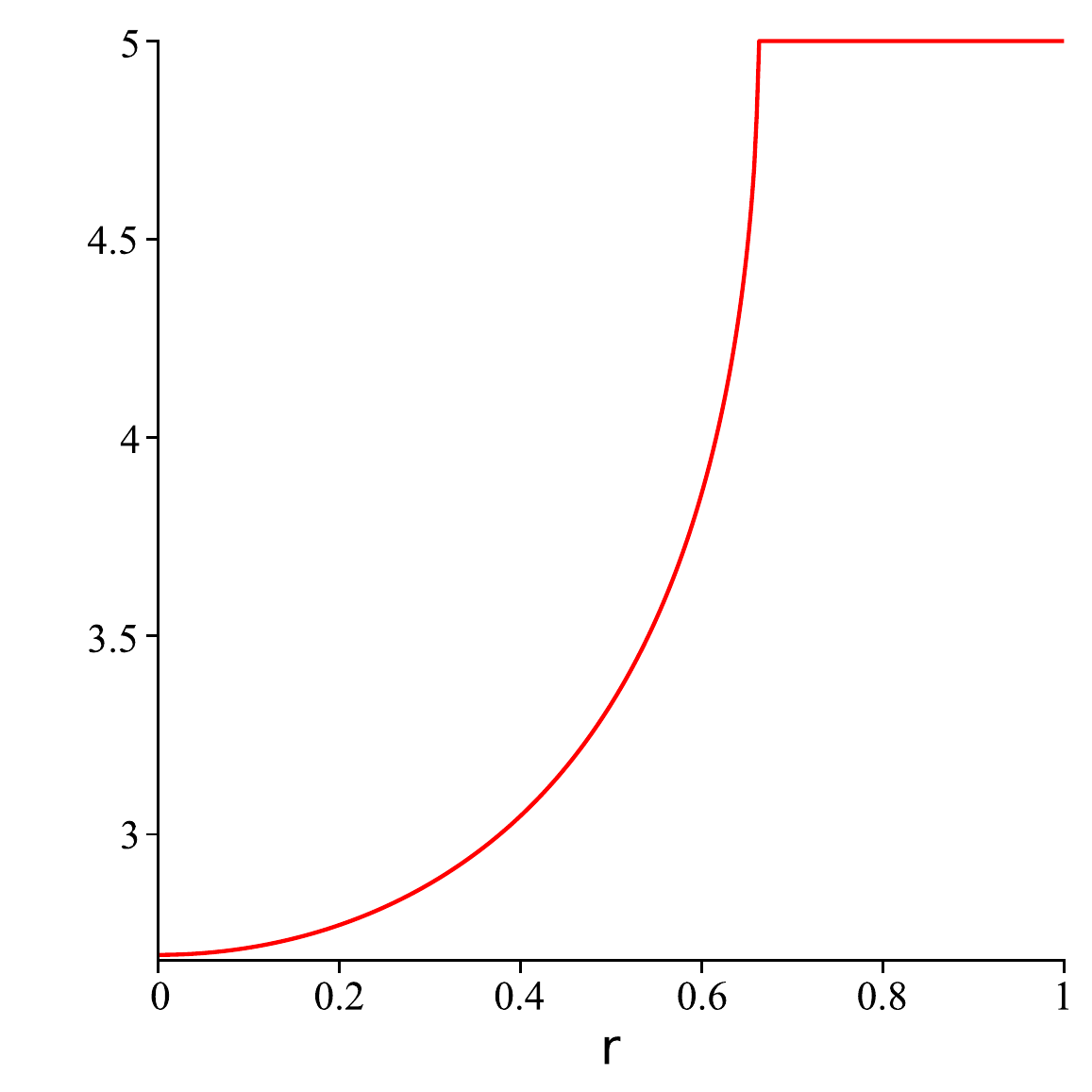}
        \caption{Case $d=2, s=1$ and $C=5$.}
        \label{Fig:radialb}
    \end{subfigure}

    \caption{Radial part of the density of $\lambda^{\sigma}$}
    \label{Fig:radial}
\end{figure}

The three-dimensional graphs of these densities are shown in Figure 3.

\begin{figure}[!htbp]
    \centering
    \begin{subfigure}[b]{0.4\textwidth}
        \centering
        \includegraphics[width=\textwidth]{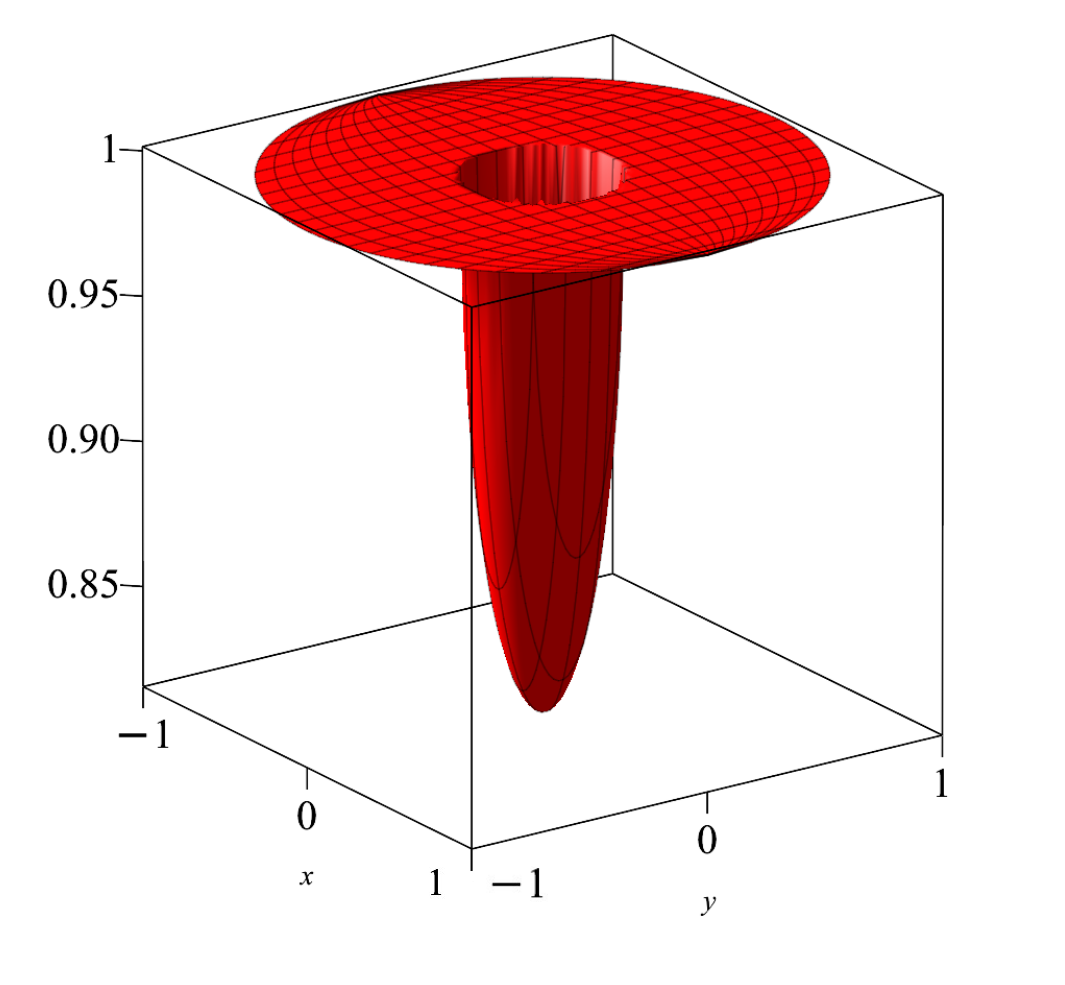}
        \caption{Case $d=2, s=1$ and $C=1$.}
        \label{Fig:3da}
    \end{subfigure}
   \hspace{0.1\textwidth}
    \begin{subfigure}[b]{0.4\textwidth}
        \centering
        \includegraphics[width=\textwidth]{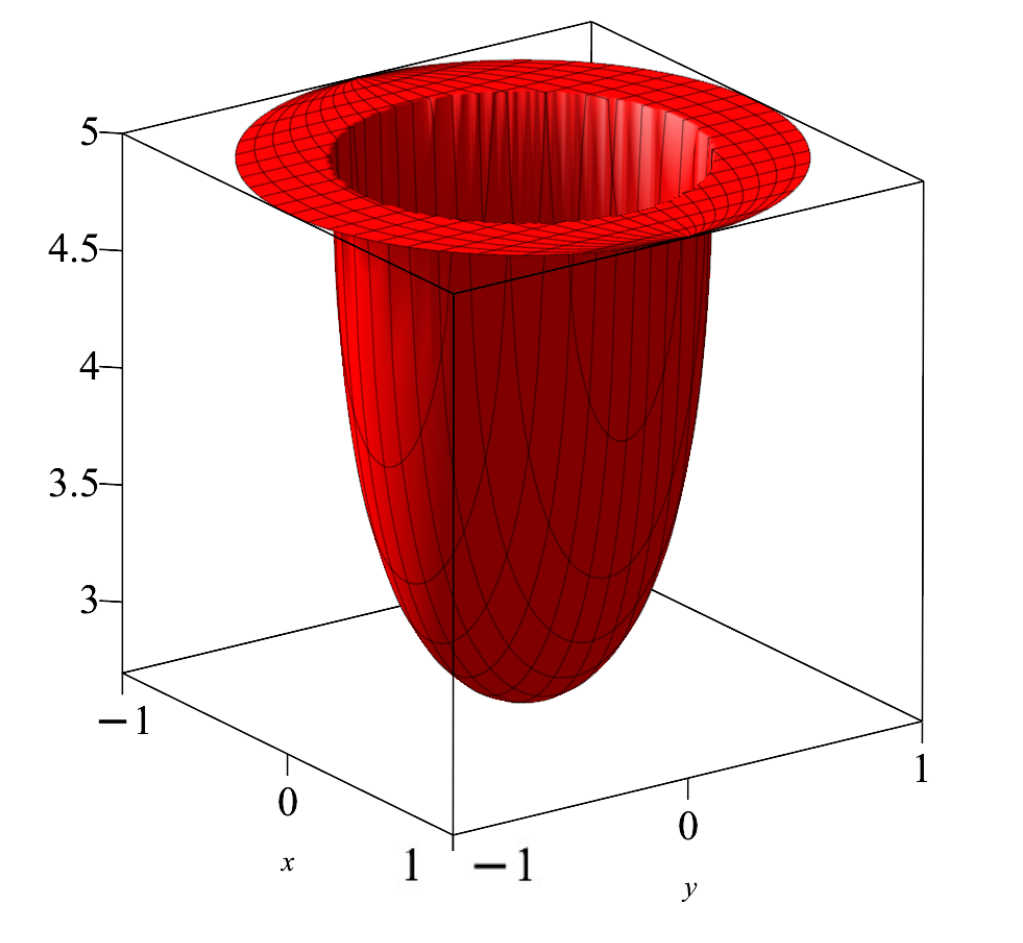}
        \caption{Case $d=2, s=1$ and $C=5$.}
        \label{Fig:3db}
    \end{subfigure}
    \caption{Three-dimensional plot of the density of $\lambda^{\sigma}$}
    \label{Fig:3d}
\end{figure}

\section{Numerical Examples. Constrained Leja Points}

In the previous section, we completely solved a constrained equilibrium problem on the Unit Ball. Unfortunately, to get such an explicit solution is not possible in general, and a numerical method to approach this constrained equilibrium measure is needed. In this sense, this section presents the numerical results obtained when computing the \textit{constrained Leja points} for the unit disk ($d=2$) and for the unit ball ($d=3$), corresponding to the equilibrium problem under the constraint \eqref{constraint}. These points are a variant of the well-known \textit{Leja} points (which in turn follow the same asymptotic distribution as that for the \textit{Fekete} points, see \cite[Ch. V]{ST24}, for example), but implementing the constraint.






To solve the problem, we followed the Constrained Leja Point algorithm used in \cite{CD01}. We illustrate separately the case of the unit disk and that of the unit ball, and consider different values of the constant $C$.
All computations were performed in Fortran 90.

\subsection{The Unit Disk}

We discretized the unit disk using a polar grid of the form 

\[(r_i, \theta_j) = \left( \sqrt{\frac{i}{n}}, \frac{2 \pi j}{m} \right),\]

\noindent where $i=1, \dots, n$, $j=1, \dots, m-1$. This choice of the $r_i$'s ensures a uniform distribution of the grid points over area and avoids clustering near the origin. For the next examples we take $s=1$. 
Now, starting with $p_0=(0,1)$, at each step we find the next Leja point $p_i$ as the grid point that minimizes the potential function $\Sigma_{j=1}^{i-1} \frac{1}{||p_i - p_j||}$, where $p_0, p_1, \dots, p_{i-1}$ are the previously selected points. 

\vspace{1\baselineskip}
\noindent It is important to point out that the choice of the grid affects the results even when using the same number of grid points, as shown in Figure~\ref{Fig5}. Observe, in addition, that the saturated region is approaching the circumference as the value of the constant $C$ increases. Compare Fig. \ref{Fig4} ($C=5$) with the graphs in Figs. \ref{Fig:radial}(B) and \ref{Fig:3d}(B) in the previous section, where $C$ is also $5$.

\begin{figure}[!htbp]
    \centering
    \begin{subfigure}[b]{0.45\textwidth}
        \centering
        \includegraphics[width=\textwidth]{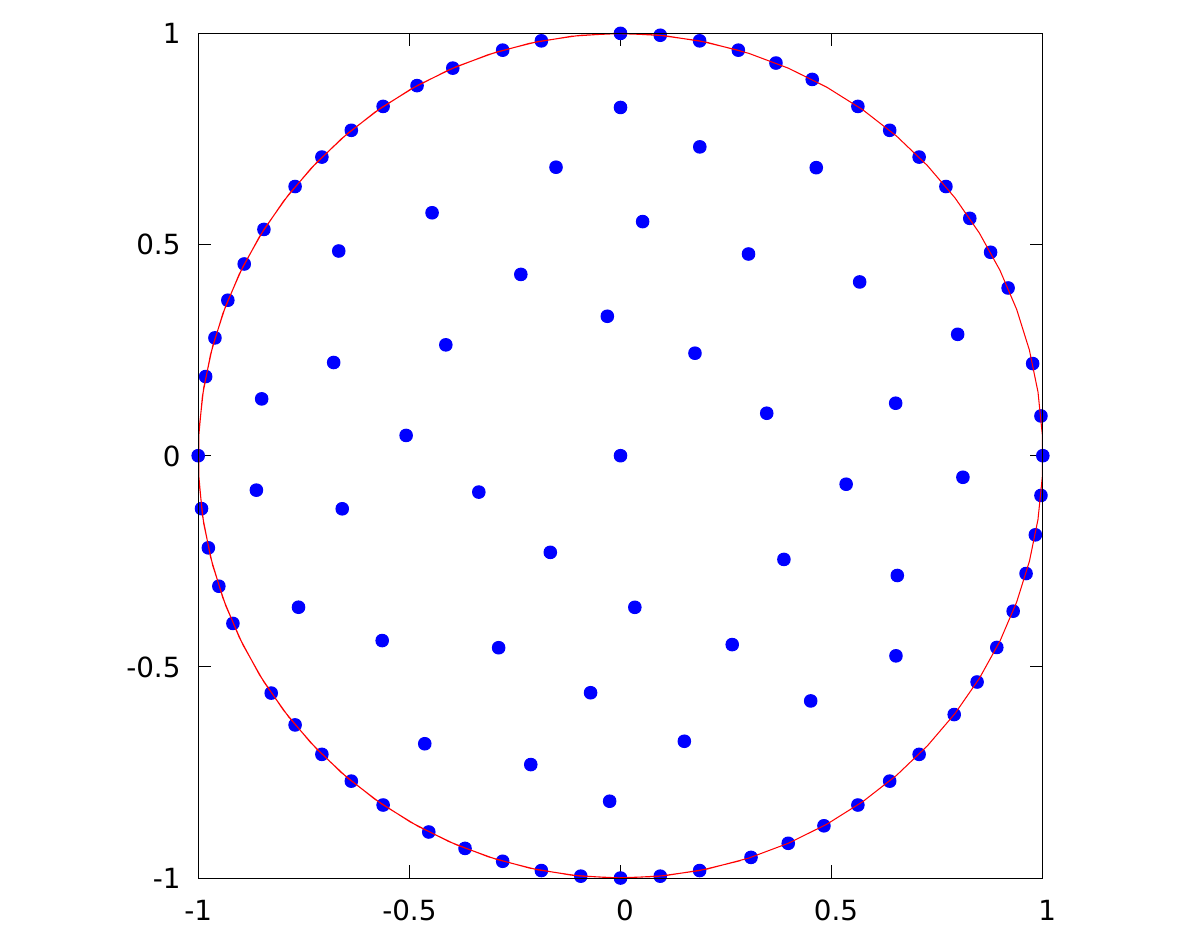}
        \caption{100 Leja points over a 20,000-point grid ($C=200$)}
        \label{Fig3a}
    \end{subfigure}
    \hfill
    \begin{subfigure}[b]{0.45\textwidth}
        \centering
        \includegraphics[width=\textwidth]{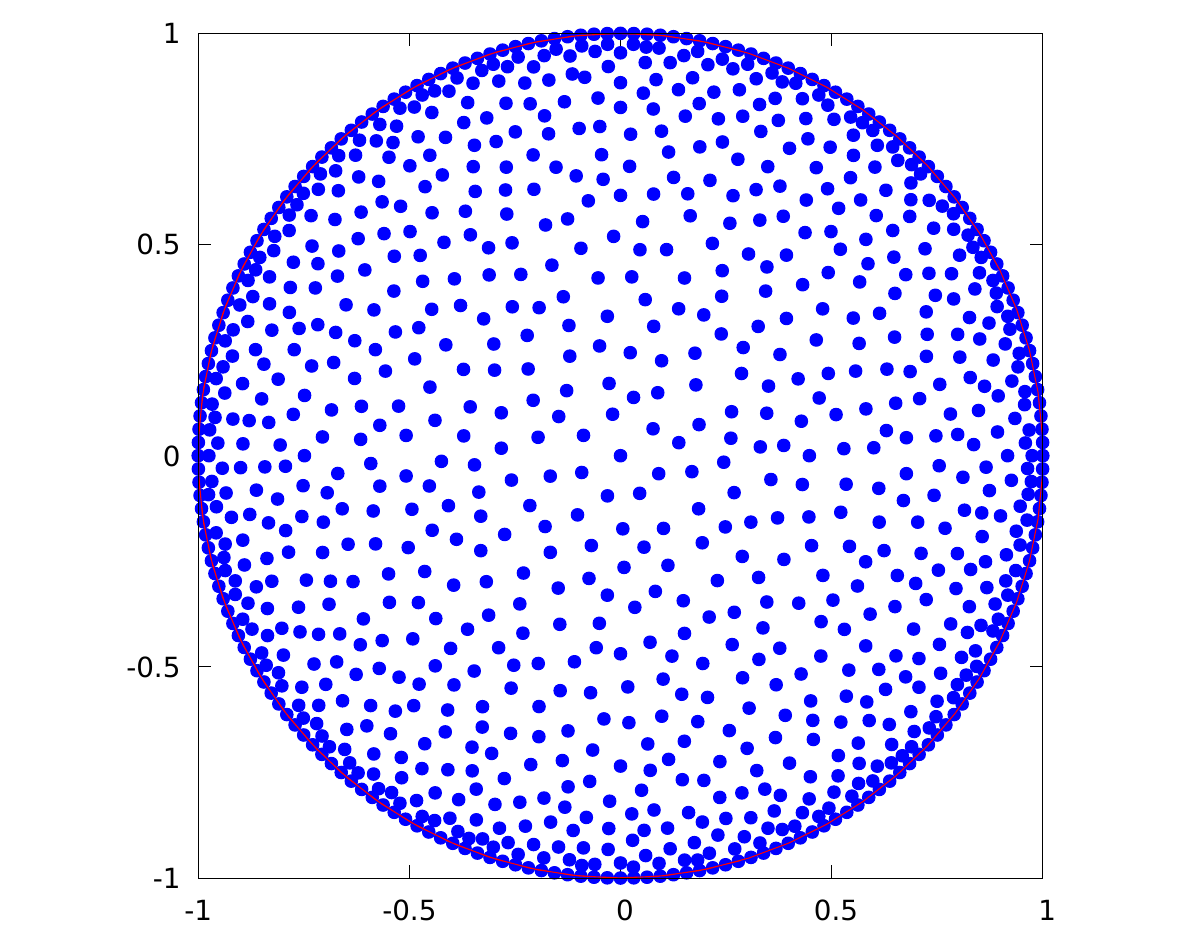}
        \caption{1,000 Leja points over a 20,000-point grid ($C=20$)}
        \label{Fig3b}
    \end{subfigure}

    \caption{Leja points in the unit disk}
    \label{Fig3}
\end{figure}

\begin{figure}[!htbp]
    \centering
    \includegraphics[width=0.5\textwidth]{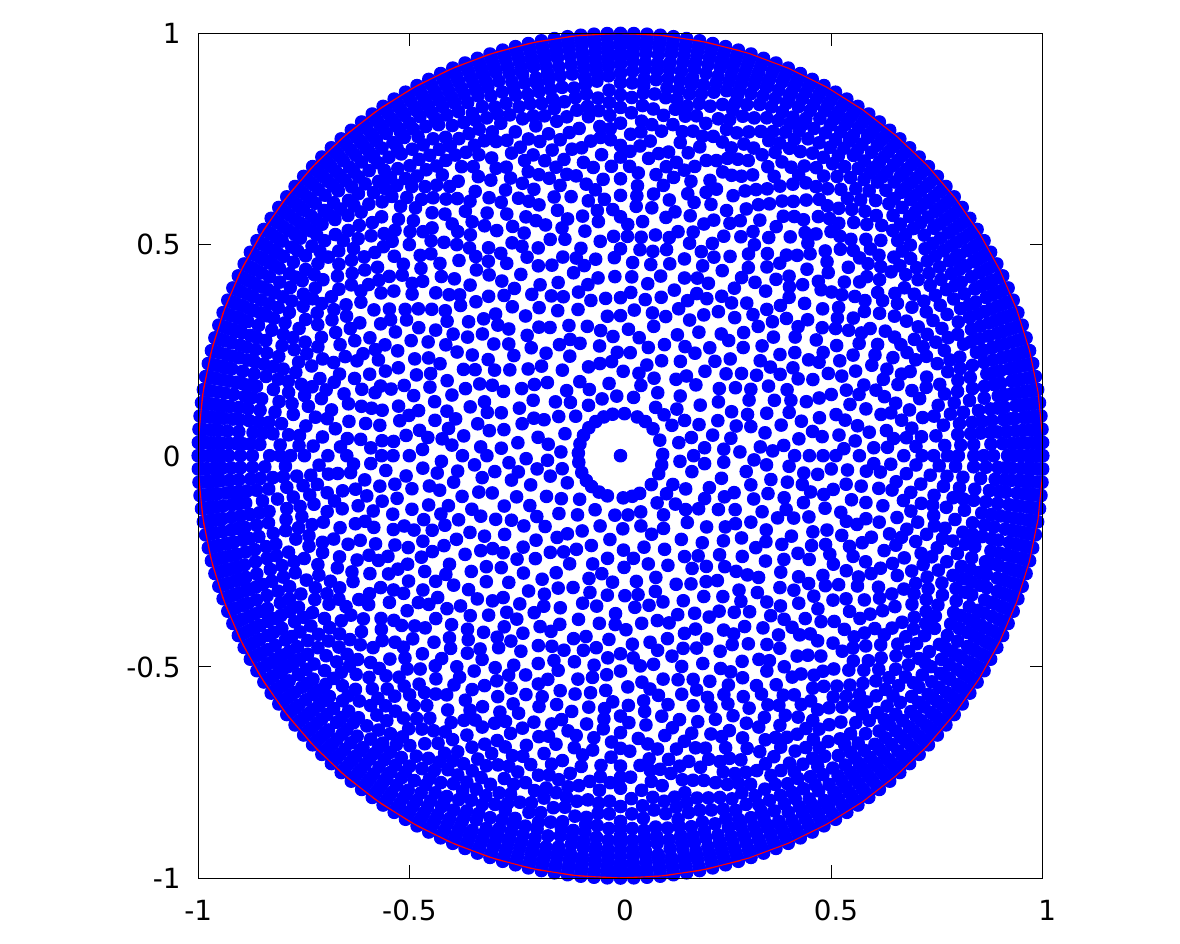}
    \caption{4,000 Leja points over a 20,000-point grid ($C=5$)}
    \label{Fig4}
\end{figure}

\begin{figure}[!htbp]
    \centering
    \begin{subfigure}[b]{0.45\textwidth}
        \centering
        \includegraphics[width=\textwidth]{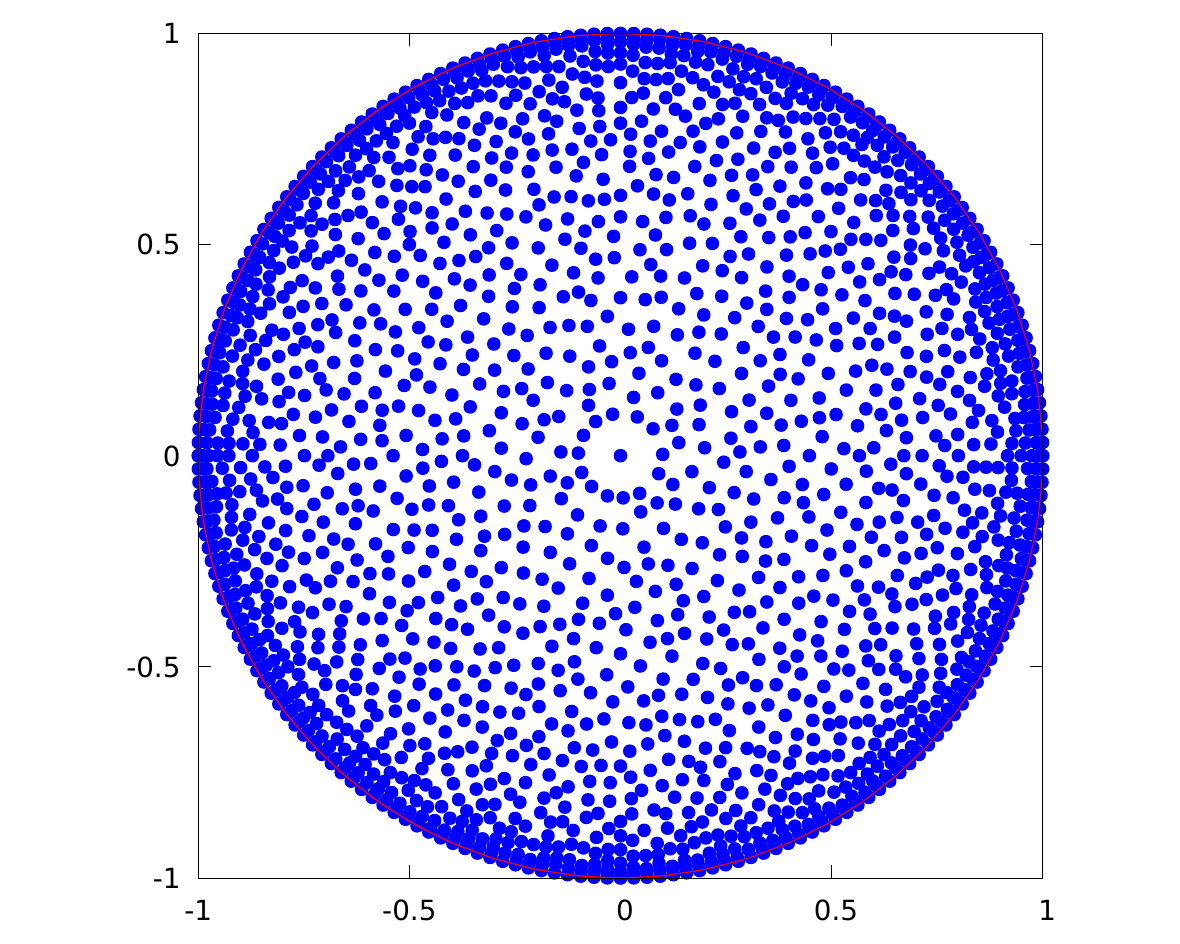}
        \caption{2,000 Leja points over a 100x200 grid}
        \label{Fig5a}
    \end{subfigure}
    \hfill
    \begin{subfigure}[b]{0.45\textwidth}
        \centering
        \includegraphics[width=\textwidth]{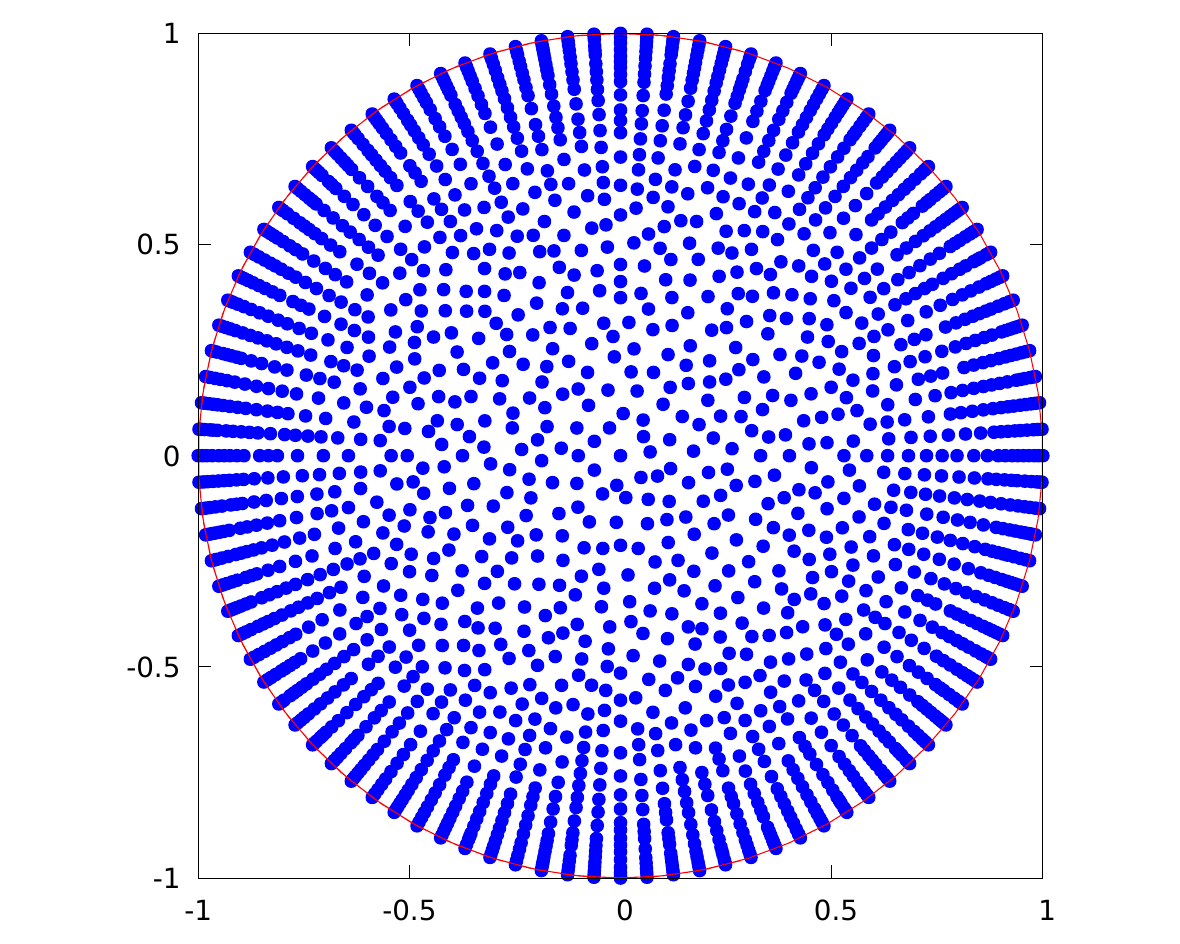}
        \caption{2,000 Leja points over a 200x100 grid}
        \label{Fig5b}
    \end{subfigure}
    \caption{Comparison of results obtained using different grids for C=10}
    \label{Fig5}
\end{figure}

\begin{remark} Some of the graphs show no Leja points near the origin. This is due to our choice of the $r_i$'s: while it ensures a uniform distribution of the grid points, it also prevents them from being within $\frac{1}{\sqrt{n}}$ of the origin. This can be easily improved by choosing large values of $n$, as seen in Figure~\ref{Fig6}, but doing so will increase the computational time. 
\end{remark}

\begin{figure}[!htbp]
    \centering
    \includegraphics[width=0.5\textwidth]{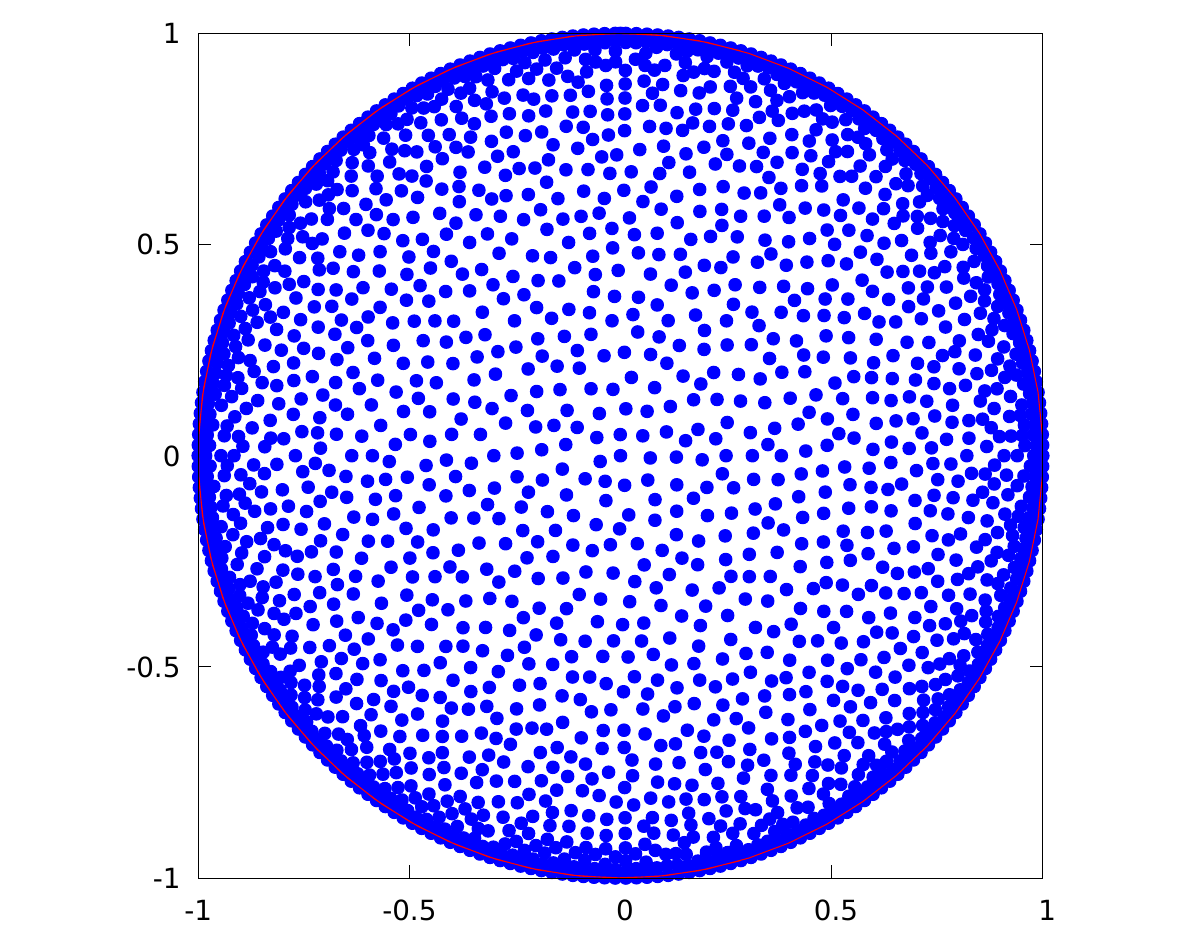}
    \caption{2,000 Leja points over a 100,000-point grid ($C=50$)}
    \label{Fig6}
\end{figure}

\subsection{The Unit 3D Ball.}

We discretized the unit ball using a modified spherical grid of the form

\[(r_i, \theta_j, \phi_k) = \left( \frac{i}{n}, \frac{2 \pi j}{m_{i,k}}, \frac{ \pi k}{p} \right),\]

\noindent where $i=1, \dots, n, j=1, \dots m_{i,k}-1, k=1, \dots, p$. Here, for a more uniform coverage, we chose the number $m_{i,k} \ge 1$ of $\theta$-samples selected on each $(r_i, \phi_k)$-ring to be proportional to the integer part of the circumference $2 \pi r_i \sin{\phi_k}$ of that ring. This avoids oversampling near the poles and undersampling near the equator. For the next example, we take the value $s=2$. Thus, starting at the North pole $p_0 = (0,0,1)$, at each step we find the next Leja point $p_i$ as the grid point that minimizes the potential function 
$\Sigma_{j=1}^{i-1} \frac{1}{||p_i - p_j||^2}$, where $p_0, p_1, \dots, p_{i-1}$ are the previously selected points.

From the examples below, it can be seen that the Leja points obtained are migrating towards the surface of the unit sphere. The choice of the grid points affects the results.   

\begin{figure}[!htbp]
    \centering
    \begin{subfigure}[b]{0.55\textwidth}
        \centering
        \includegraphics[width=\textwidth]{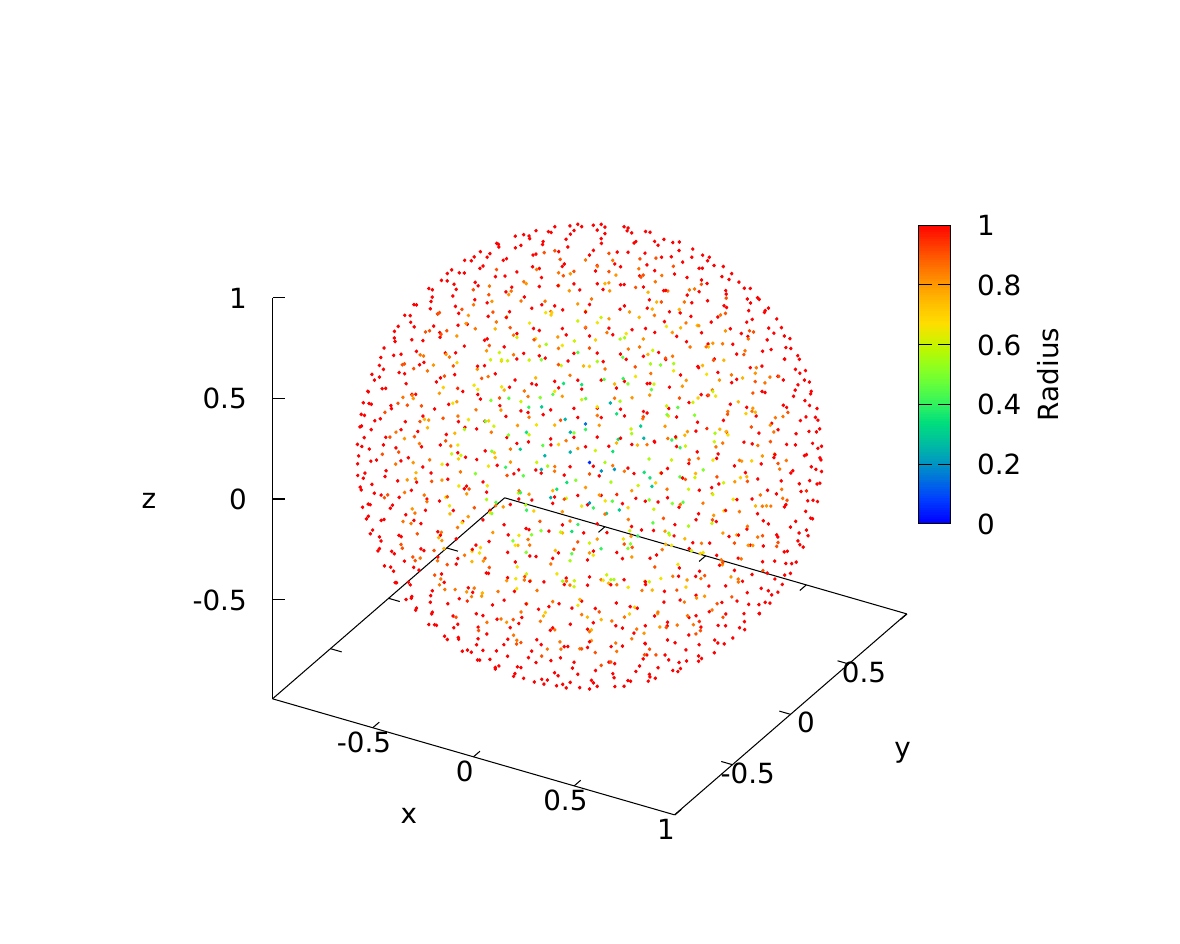}
        \caption{Spatial distribution}
        \label{Fig7a}
    \end{subfigure}
    \hfill
    \begin{subfigure}[b]{0.4\textwidth}
        \centering
        \includegraphics[width=\textwidth]{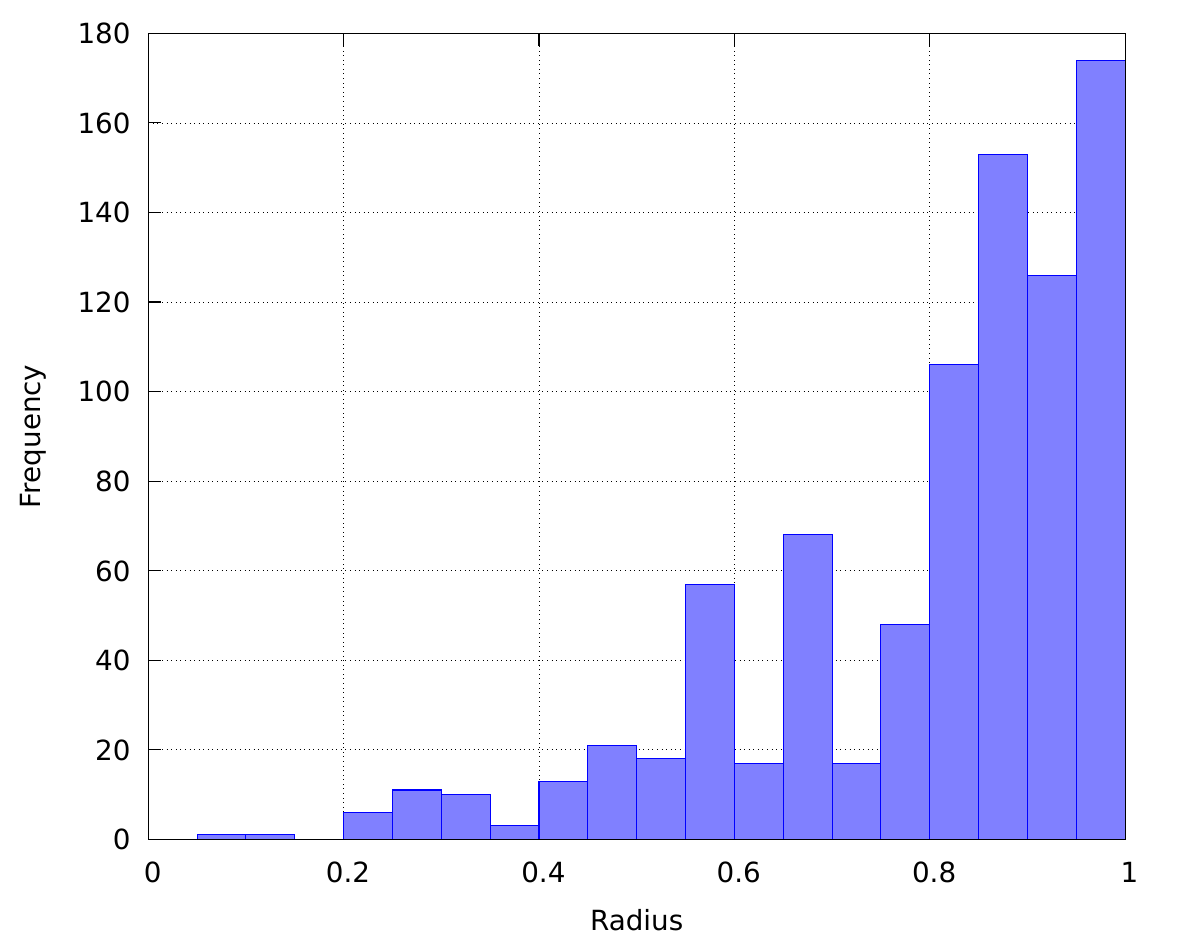}
        \caption{Histogram of radii}
        \label{Fig7b}
    \end{subfigure}
    \caption{1,500 Leja points over 20,762-point grid ($C=13.8$)}
    \label{Fig7}
\end{figure}

\begin{figure}[!htbp]
    \centering
    \begin{subfigure}[b]{0.55\textwidth}
        \centering
        \includegraphics[width=\textwidth]{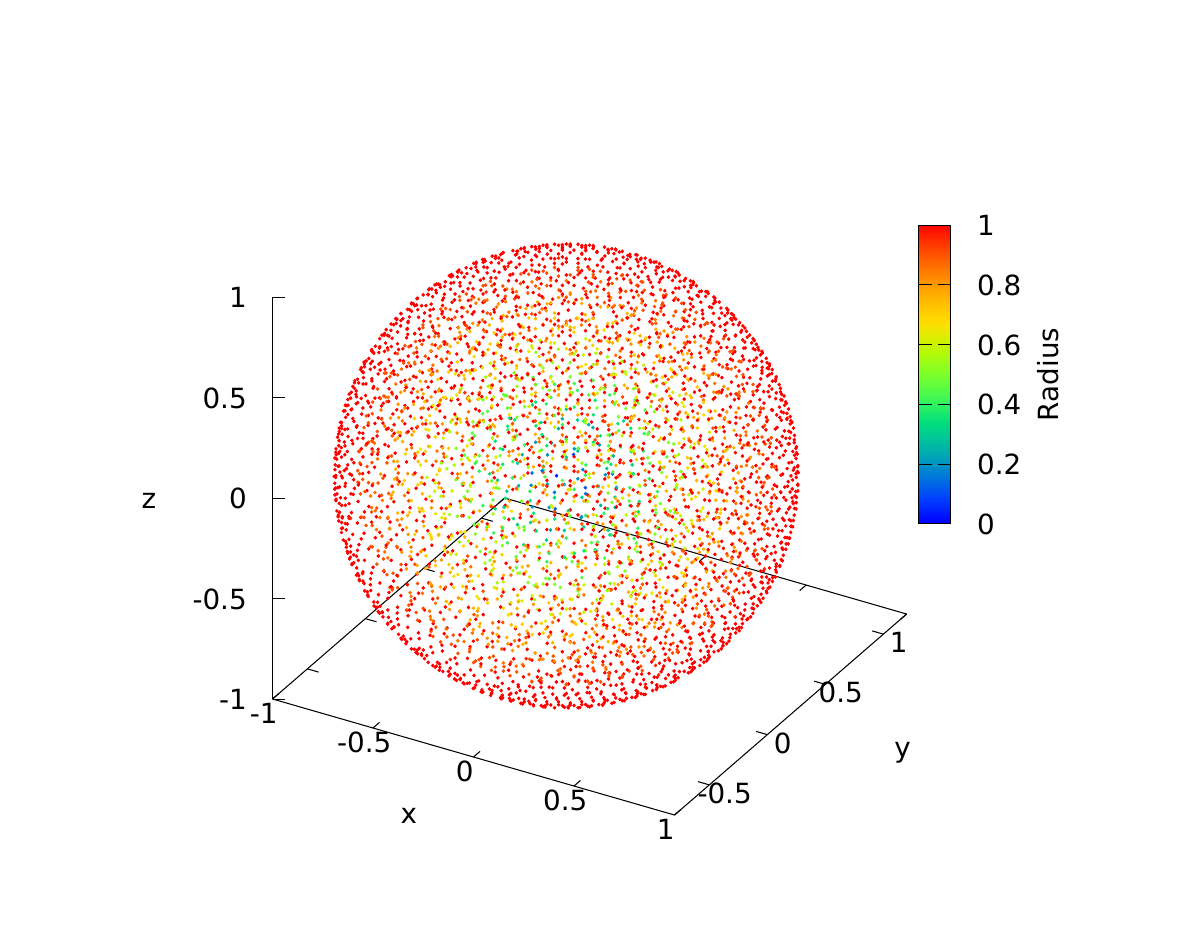}
        \caption{Spatial distribution}
        \label{Fig8a}
    \end{subfigure}
    \hfill
    \begin{subfigure}[b]{0.4\textwidth}
        \centering
        \includegraphics[width=\textwidth]{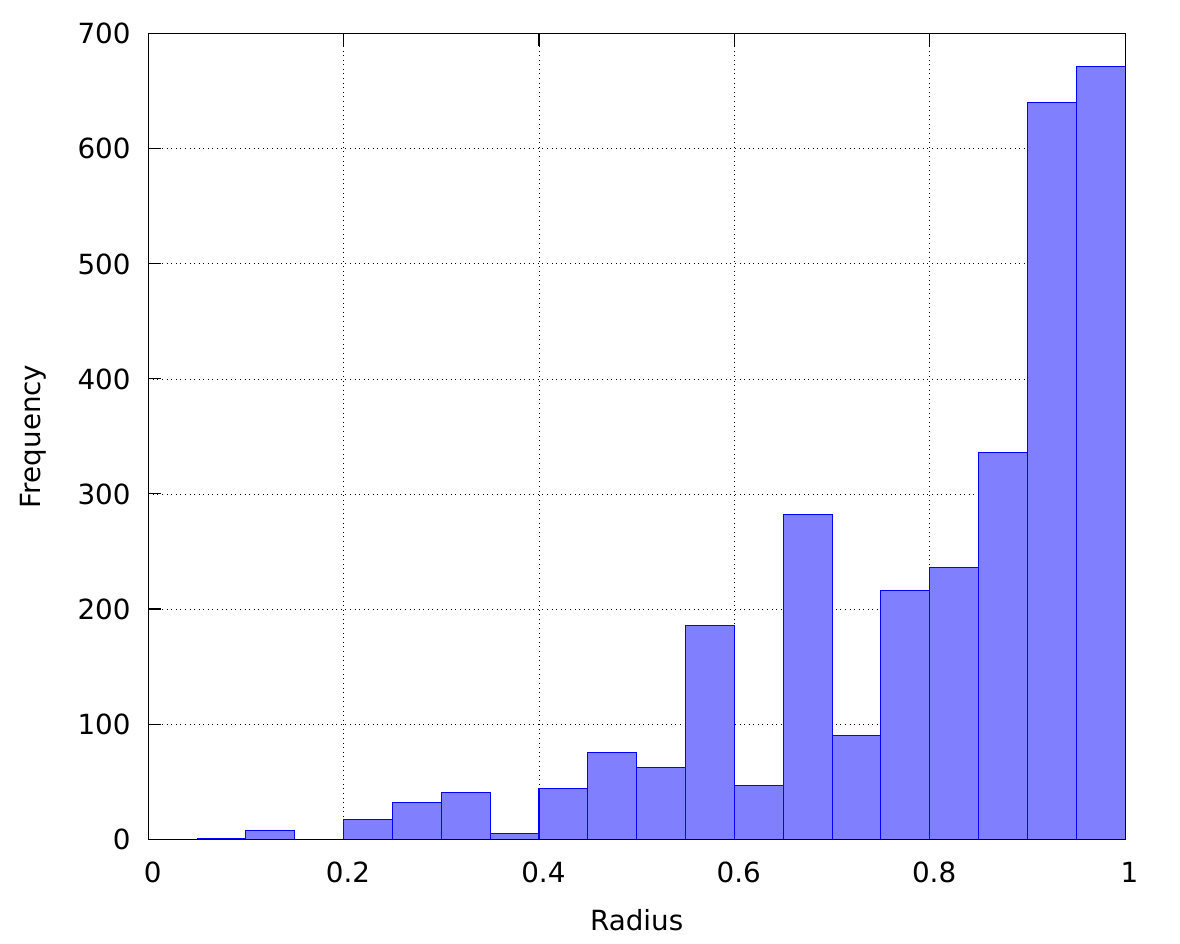}
        \caption{Histogram of radii}
        \label{Fig8b}
    \end{subfigure}
    \caption{5,000 Leja points over 41,582-point grid ($C=8.3$)}
    \label{Fig8}
\end{figure}

\begin{figure}[!htbp]
    \centering
    \includegraphics[width=0.75\textwidth]{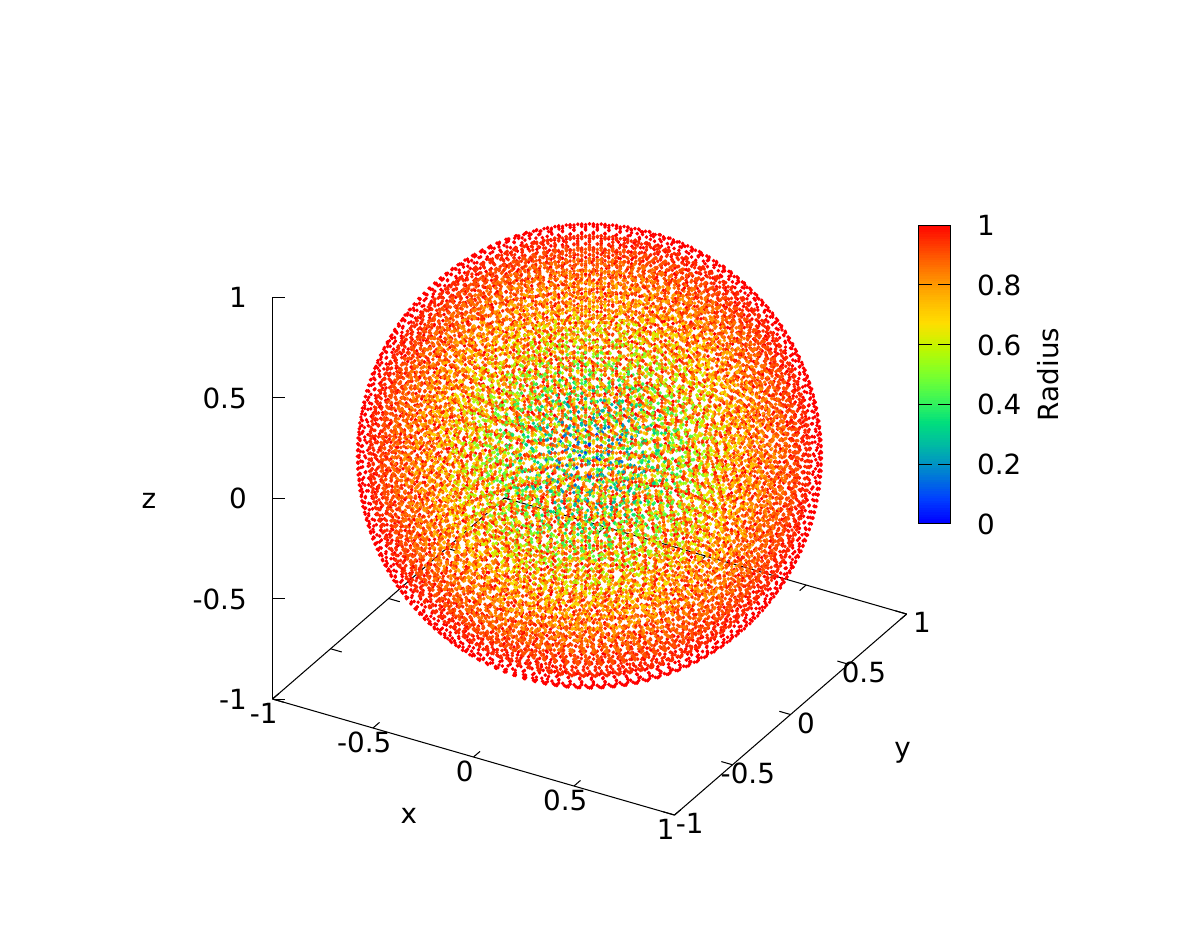}
    \caption{20,791 Leja points over a 41,582-point grid ($C=2$)}
    \label{Fig9}
\end{figure}

\FloatBarrier    

\end{document}